\documentclass[12pt]{article}
\usepackage{amssymb,amsfonts,amsthm}
\usepackage{amsmath,amsthm,amssymb}
\usepackage{amscd}
\usepackage{amsfonts}
\usepackage{color}
\newcommand{\con}{\hbox{Cont}}
\newcommand{\ocs}{\hbox{OccSet}}
\newcommand{\non}{\hbox{Non}}
\newcommand{\lin}{\hbox{Lin}}

\newtheorem{theorem}{Theorem}[section]
\newtheorem{ex}[theorem]{Example}
\newtheorem{cor}[theorem]{Corollary}

\newtheorem{fact}[theorem]{Fact}
\newtheorem{lemma}[theorem]{Lemma}
\newtheorem{definition}[theorem]{Definition}

\newtheorem{question}{Question}
\newtheorem{claim}{Claim}

\begin{document}

\title{Non-finitely based monoids}
\author{ Olga Sapir }

\maketitle

\begin{abstract} We present a general method for proving that a semigroup is non-finitely based. The method is strong enough to cover the non-finite
basis arguments in articles \cite{BS, Is, MJ1,MJ,JS, EL2, EL5, LZL,P,OS, Sh, Tr2, WTZ1, WTZ}. In particular,
the method allows to generalize the results in \cite{ BS, EL2, WTZ1, WTZ} and to simplify their proofs.  The method also allows to remove one of the requirements on the ``special system of identities" used by P. Perkins in \cite{P} to
 find the first two examples of finite non-finitely based semigroups. We use our method to prove eleven new sufficient conditions under which a monoid is non-finitely based. As an application, we find infinitely many new examples of finite finitely based aperiodic monoids whose direct product is non-finitely based.
\end{abstract}


\section{Introduction}

An algebra  is said to be {\em finitely based} (FB) if there is a finite subset of its identities from which all of its identities may be deduced.
Otherwise, an algebra is said to be {\em non-finitely based} (NFB).
The famous Tarski's Finite Basis Problem asks if there is an algorithm to decide when a finite algebra is finitely based.
In 1996, R. McKenzie  \cite{RM} solved this problem in the negative showing that the classes of FB and inherently not finitely based finite algebras are recursively inseparable. (A locally finite algebra is said to be {\em inherently not finitely based} (INFB) if any locally finite variety containing it is NFB.)
It is still unknown whether the set of FB finite semigroups is recursive although a very large volume of work is devoted to this problem (see
the surveys \cite{SV,MV}). In contrast with McKenzie's result, a powerful description of the INFB finite semigroups has been obtained by M. Sapir \cite{MS, MS1}.

In 1968, P. Perkins \cite{P} found the first sufficient condition under which a monoid (semigroup with an identity element) is NFB.
By using this condition, he constructed the first two examples of finite NFB semigroups.  The first example was the 6-element Brandt monoid and the second example was the 25-element monoid obtained from the set of words
$W= \{abtba, atbab, abab, aat\}$ by using the following construction attributed to Dilworth.

 Let ${\mathfrak A}$ be an alphabet and $W$ be a set of words in the free monoid ${\mathfrak A}^*$. Let $S(W)$ denote the Rees quotient  over the ideal of  ${\mathfrak A}^*$ consisting of all words that are not subwords of words in $W$. For each set of words $W$, the semigroup $S(W)$ is a monoid with zero whose nonzero elements are the subwords of words in $W$. Evidently, $S(W)$ is finite if and only if $W$ is finite.

A word ${\bf u}$ is said to be an isoterm \cite{P} for a semigroup $S$ if $S$ does not satisfy any nontrivial identity of the form ${\bf u} \approx {\bf v}$.
According to M. Sapir \cite{MS}, a finite semigroup $S$ is INFB  if and only if every Zimin word (${\bf Z}_1=x_1, \dots, {\bf Z}_{k+1} = {\bf Z}_kx_{k+1}{\bf Z}_k, \dots$)  is an isoterm for $S$. This result implies that a monoid of the form $S(W)$ is never INFB while the Brandt monoid is INFB.

This article is the first part of a sequence of four submissions.  Two earlier long submissions were re-organised into the present four separate articles to aid focus to the presentation. The next article \cite{OS3} contains a method for proving that a semigroup is finitely based.
In articles \cite{OS1, OS2} we study the following problem.

 \begin{question}\cite[M. Sapir]{SV} \label{qMS} Is the set of finite monoids of the form $S(W)$ recursive?
\end{question}

If a variable $t$ occurs exactly once in a word ${\bf u}$ then we say that $t$ is {\em linear} in ${\bf u}$. If a variable $x$ occurs more than once in a word ${\bf u}$ then we say that $x$ is {\em non-linear} in ${\bf u}$. In article \cite{OS1}, we show how to recognize FB semigroups among the monoids of the form $S(W)$ where $W$ consists of a single word with at most two non-linear variables. In article \cite{OS2}, we show how to recognize FB semigroups among the monoids of the form $S(W)$
with some other natural restrictions on the set $W$.

We say that a semigroup $S$ is {\em non-finitely based by way of a set of identities $\Sigma$} if $S$ satisfies all the identities in $\Sigma$ but $\Sigma$ cannot be derived from any finite set of identities of $S$. Recall that {\em axiomatic rank of an algebra} is the minimal number of variables $m$ such that all of its identities may be deduced from its identities in $m$ variables. It is well-known that every locally finite variety of algebras of finite axiomatic rank is finitely based. This implies that if a locally finite semigroup is non-finitely based by way of a set of identities $\Sigma$ then there is no bound on the number of variables involved in the identities from $\Sigma$.

Let $\sim_S$ denote the fully invariant congruence on the free semigroup $\mathfrak A^+$ corresponding to a semigroup $S$.
 The sufficient condition found by P. Perkins in \cite{P} (see Theorem \ref{PSC} below) exhibits a certain set of words $W$ with two non-linear variables, a certain set of identities $\Sigma$ and states the following:

 ($\mathcal P$) If a monoid $S$ satisfies all identities in $\Sigma$ and the words in $W$ are $\sim_S$-related to other words in $\mathfrak A^+$ in a certain way, then the monoid $S$ is non-finitely based by way of $\Sigma$.

Studying Question \ref{qMS} in articles \cite{MJ1,MJ,JS,LZL,OS} has resulted in discovering many new sufficient conditions under which a monoid is non-finitely based. Unlike the ``Zimin words" condition of M. Sapir \cite{MS}, all  conditions in these articles are of the same form ($\mathcal P$) as the original Perkins condition.

In this article we present a general method (see Lemma \ref{nfblemma1} below) for proving that a semigroup is non-finitely based.
We use this method to find eleven new sufficient conditions of form ($\mathcal P$) under which a monoid is non-finitely based (see Theorems \ref{SL1}, \ref{nfbpairs}, \ref{BSnew} and \ref{BSnew1} below). We also use our method to reduce the number of requirements in Perkins' \cite{P} and Lee's \cite{EL2} sufficient conditions under which a semigroup is non-finitely based (see Theorems \ref{PSC} and \ref{EL} below).

 We observe that Lemma \ref{nfblemma1} is strong enough to cover the non-finite
basis arguments in articles \cite{BS, Is, MJ1,MJ,JS, EL2, EL5, LZL,P,OS, Sh, Tr2, WTZ1, WTZ}. In particular, Lemma \ref{nfblemma1} can be used to verify all
existing sufficient conditions of form $\mathcal P$ under which a semigroup is non-finitely based.
We also observe that in every existing sufficient condition of form ($\mathcal P$), all words involved in set $W$ contain at most two non-linear variables.

Recently,  E. Lee suggested to investigate the finite basis properties of the six-element semigroup $L= \langle a,b \mid aa=a, bb=b, aba=0 \rangle$ and of the monoid $L^1$ obtained by adjoining an identity element to $L$.  W. Zhang and Y. Luo proved in \cite{WTZ1} that the semigroup $L$ is non-finitely based and  W. Zhang proved in \cite{WTZ} that the monoid $L^1$ is also non-finitely based. A semigroup is called {\em aperiodic} if it contains only trivial subgroups. The semigroups  $L$ and $L^1$ are the first examples of finite aperiodic NFB semigroups for which the word $xtx$ is not an isoterm. In particular, neither of these semigroups is INFB. In this article, we use Lemma \ref{nfblemma1} to obtain short proofs that the semigroup $L$ and the monoid $L^1$ are non-finitely based. (See Corollaries \ref{WT} and \ref{WT1} below.)

The semigroup $L$ is also interesting because it is the last example of an NFB semigroup of order six. Only four out of 15 973 distinct six-element semigroups are non-finitely based \cite{LLZ} while every semigroup with five or fewer elements is finitely based \cite{Tr1, Tr}, see also  \cite{EL}.
In \cite{MJ}, M. Jackson proved that the varieties generated by $S(\{at_1abt_2b\})$ and  $S(\{abt_1at_2b, at_1bt_2ab\})$ are {\em  limit varieties}
in a sense that each of these varieties is NFB while each proper monoid subvariety of each of these varieties is FB. Since the word $xtx$ is not an isoterm for $L^1$,
the variety generated by $L^1$ contains neither $S(\{at_1abt_2b\})$ nor  $S(\{abt_1at_2b, at_1bt_2ab\})$. Consequently,
the result of Zhang implies that there exists a new limit variety of aperiodic monoids. This gives the affirmative answer to a question of Jackson from \cite{MJ}.

Theorem  \ref{nfbpairs} consists of eight sufficient conditions of form ($\mathcal P$) under which a monoid is non-finitely based. Each of these conditions is encoded in a separate row of Table \ref{classes}.
Throughout this article, elements of a countable alphabet $\mathfrak A$ are called {\em variables} and elements of the free semigroup $\mathfrak A^+$ are called {\em words}. If $\{x_1, x_2, \dots, x_n \}$ is a set of variables then we denote $[Xn]= x_1x_2 \dots x_n$ and $[nX] = x_nx_{n-1}\dots x_1$. The part of Theorem \ref{nfbpairs} that is encoded in the first row of Table \ref{classes} can be decoded into the following statement:
if $S$ is a monoid that satisfies the identity $xx[Yn][nY] \approx [Yn][nY]xx$ for each $n>1$ and the word $xyyx$ is an isoterm for $S$ then $S$ is non-finitely based.
This sufficient condition for the non-finite basis property of monoids can be deduced easily from the proofs of some interesting results obtained in 1970's independently by J. Isbell and L. Shneerson.

In 1970, J. Isbell \cite{Is} proved that the variety of groups  defined by the identity $xxyy \approx yyxx$ generates the variety of monoids $\mathcal M$ which is non-finitely based by way of the set of identities $\{xx[Yn][nY] \approx [Yn][nY]xx \mid n>1\}$. His proof was based on the fact that the identity $xxyy \approx yyxx$ is the only nontrivial identity whose left-hand side has length at most 4 and which holds in the group of integers $\mathbb Z$ and in the symmetric group $S_3$ simultaneously. This property of the monoid $\mathbb Z \times S_3$ is equivalent to the property that the word $xyyx$ is an isoterm for $\mathbb Z \times S_3$ and consequently, for $\mathcal M$.

In 1972,  L. Shneerson (see \cite{Sh} for exact references) described all semigroups and monoids with one defining relation satisfying nontrivial identities and noticed that the monoid $M= \langle a,b \mid abba =1 \rangle$ represents a unique (up to isomorphism) example of a one-relator monoid which is a non-cyclic group without free submonoids of rank 2. A couple of years later, he proved that the monoid $M$ is non-finitely based by way of the set of identities $\{xx[Yn][nY] \approx [Yn][nY]xx \mid n>1\}$. An examination of Shneerson's proof also shows that the only property of the monoid $M$ which is responsible for $M$ being non-finitely based by way of the set of identities $\{xx[Yn][nY] \approx [Yn][nY]xx \mid n>1\}$, is that the word  $xyyx$ is an isoterm for $M$.

The mentioned results of Shneerson together with his other results about monoids with one defining relation were adducted in \cite{Sh}.
It is proved there that if a semigroup or monoid with one defining relation has finite axiomatic rank then it is finitely based.
Article \cite{Sh} contains another interesting property of the NFB monoid $M$: it is finitely based as a group.  (Coincidentally, the NFB semigroup L is finitely based as an involution semigroup \cite{EL4}.) Also, the article \cite{Sh} contains the first example of an NFB semigroup $S$ such that the monoid $S^1$ obtained by adjoining an identity element to $S$ is finitely based. Recently, E. Lee \cite{EL3} found a way to construct finite semigroups with this property.

\section {A method for proving that a semigroup is non-finitely based}

If $S$ is a semigroup and $\bf u$ is a word then we use $[\![{\bf u}]\!]_S$ to denote the equivalence class of the fully invariant congruence $\sim_S$ on $\mathfrak A^+$ containing $\bf u$. The following obvious statement gives us a method for proving that a semigroup is non-finitely based.

\begin{fact} \label{nfblemma}
Let $S$ be a semigroup and $\{{\bf U}_n = {\bf V}_n \mid n>10 \}$ be a set of identities of $S$ in unbounded number of variables.

Assume that for each $n >10$, the word ${\bf U}_n$ contains at least $n$ variables and has some property $P_n$
such that the word ${\bf V}_n$ does not have the property $P_n$.

Let $\bf U$ be an arbitrary  word in $[\![{\bf U}_n ]\!]$ which has the property $P_n$.
Suppose that an application of a ``short" (say, in less than $n/3$ variables) identity of $S$  to the word $\bf U$ preserves the property $P_n$.
Then $S$ is non-finitely based by way of $\{{\bf U}_n = {\bf V}_n \mid n> 10\}$.
\end{fact}

We use the word {\em substitution} to refer to the homomorphisms of the free semigroup and of the free monoid. Since every substitution $\Theta$ is uniquely determined by its values on the letters of the alphabet $\mathfrak A$, we write $\Theta: \mathfrak A \rightarrow \mathfrak A^+$ if $\Theta$ is a homomorphism of the free semigroup $\mathfrak A^+$ and we write $\Theta: \mathfrak A \rightarrow \mathfrak A^*$ if $\Theta$ is a homomorphism of the free monoid $\mathfrak A^*$.

\begin{cor} \label{nfbcor} Let $S$ be a semigroup. Suppose that for each $n$ large enough one can find a word ${\bf U}_n$ in at least $n$ variables such that ${\bf U}_n$ is not an isoterm for $S$ but every word $\bf u$ in less than $n/2$ variables is an isoterm for $S$ whenever $\Theta({\bf u})={\bf U}_n$ for some
substitution $\Theta: \mathfrak A \rightarrow \mathfrak A ^+$. Then $S$ is non-finitely based.
\end{cor}

\begin{proof} Fix large enough number $n$. Since ${\bf U}_n$ is not an isoterm for $S$, the semigroup $S$ satisfies some non-trivial identity  ${\bf U}_n \approx {\bf V}_n$. Let $P_n$ be the property of being equal to the word ${\bf U}_n$. Let $\bf U$ be an arbitrary  word in $[\![{\bf U}_n ]\!]$ which has the property $P_n$. Let ${\bf u} \approx {\bf v}$ be an identity in less than $n/3$ variables such that the word $\bf u$ is applicable to $\bf U$. Then for some substitution $\Theta: \mathfrak A \rightarrow \mathfrak A ^+$, the word $\Theta({\bf u})$ is a subword of $\bf U$. If the word $\bf u$ does not depend on $t_1, t_2 \in \mathfrak A$ then we may assume that $\Theta(t_1{\bf u}t_2) = {\bf U} = {\bf U}_n$. Since the word $t_1{\bf u}t_2$ is an isoterm for $S$, the semigroup $S$ is non-finitely based by Fact \ref{nfblemma}.
\end{proof}

 Corollary \ref{nfbcor} is behind many existing non-finite basis arguments. For example, let $A_2^1$ denote the monoid obtained by adjoining an identity element to the semigroup $A_2= \langle a,b \mid aba=aa=a, bab=b, bb=0 \rangle$ of order five.
The monoid $A_2^1$ is INFB by the ``Zimin words" condition of M. Sapir \cite{MS}. In \cite{Tr2}, A. Trahtman independently proved that $A_2^1$ is non-finitely
based by way of the set of identities $\{[Xn]y[nX]y[Xn] \approx [Xn]y[nX]y[Xn]y[nX]y[Xn] \mid n>1\}$. He proved it by showing that
every word $\bf u$ in less than $n/3$ variables is an isoterm for $A_2^1$ whenever $\Theta({\bf u})= [Xn]y[nX]y[Xn]$ for some
substitution $\Theta: \mathfrak A \rightarrow \mathfrak A ^+$.

 However, Corollary \ref{nfbcor} is not suitable for explaining why the semigroups $L$ and $L^1$ mentioned in the introduction are NFB, because the word $xtx$ is not an isoterm for either of these semigroups. Now we introduce some language suitable for describing possible properties $P_n$ which one might consider while using
 Fact \ref{nfblemma}.

 If some variable $x$ occurs $n \ge 0$ times in a word ${\bf u}$ then we write $occ_{\bf u}(x)=n$ and say that $x$ is {\em $n$-occurring} in ${\bf u}$.
The set $\con({\bf u}) = \{x \in \mathfrak A \mid occ_{\bf u}(x)>0 \}$ of all variables contained in a word ${\bf u}$ is called the {\em content of ${\bf u}$}.

We use $_{i{\bf u}}x$ to refer to the $i^{th}$ from the left occurrence of $x$ in ${\bf u}$. We use $_{last {\bf u}}x$ to refer to the last occurrence of $x$ in ${\bf u}$.  The set $\ocs({\bf u}) = \{ {_{i{\bf u}}x} \mid x \in \mathfrak A, 1 \le i \le occ_{\bf u} (x) \}$ of all occurrences of all variables in ${\bf u}$ is called the {\em occurrence set of ${\bf u}$}. The word ${\bf u}$ induces a (total) order $<_{\bf u}$ on the set $\ocs({\bf u})$ defined by ${_{i{\bf u}}x} <_{\bf u} {_{j{\bf u}}y}$ if and only if the $i^{th}$ occurrence of $x$ precedes the $j^{th}$ occurrence of $y$ in ${\bf u}$. For example, $\ocs(xyx) = \{ {_{1xyx}x}, {_{1xyx}y}, {_{2xyx}x} \}$ together with the order $<_{xyx}$ is a structure $(\ocs(xyx), <_{xyx})$ isomorphic to the three-element totally ordered set: $({_{1xyx}x}) <_{xyx} ({_{1xyx}y}) <_{xyx} ({_{2xyx}x})$.

\begin{definition} \label{trans}

We say that a set $\Omega$ is a {\em good collection of maps} if each element $f_{{\bf u}, {\bf v}} \in \Omega$ is an injection from a subset of $\ocs({\bf u})$ into the set $\ocs({\bf v})$ of some identity $({\bf u} \approx {\bf v}) \in \mathfrak A^+ \times \mathfrak A^+$ and each of the following conditions is satisfied:

(i) for each $({\bf u} \approx {\bf v}) \in \mathfrak A^+ \times \mathfrak A^+$ there is at most one map  $f_{{\bf u}, {\bf v}} \in \Omega$;

(ii) if $f_{{\bf u}, {\bf u}} \in \Omega$ then $f_{{\bf u}, {\bf u}}$ is the identity map on its domain;

(iii) if the composite $f_{{\bf u}, {\bf v}} \circ f_{{\bf v}, {\bf w}}$
is defined on $c \in \ocs({\bf u})$, then $f_{{\bf u}, {\bf w}} \in \Omega$ and $f_{{\bf u}, {\bf w}} (c) = (f_{{\bf u}, {\bf v}} \circ f_{{\bf v}, {\bf w}}) (c)$.

\end{definition}

 If $X \subseteq \ocs({\bf u})$ and $f_{{\bf u}, {\bf v}}$ is an injection from a subset of $\ocs({\bf u})$ into the set $\ocs({\bf v})$
 then we say that the set $X$ is {\em $f_{{\bf u}, {\bf v}}$-stable} in an identity ${\bf u} \approx {\bf v}$
if the map $f_{{\bf u}, {\bf v}}$ is defined on $X$ and is an isomorphism of the (totally) ordered sets $(X, <_{\bf u})$ and $(f_{{\bf u}, {\bf v}}(X), <_{\bf v})$.
Otherwise, we say that the set $X$ is {\em $f_{{\bf u}, {\bf v}}$-unstable} in ${\bf u} \approx {\bf v}$.

\begin{fact} \label{gcm} Let $\Omega$ be a good collection of maps. Then the following is true:

if $f_{{\bf u}, {\bf v}}, f_{{\bf v}, {\bf w}}, f_{{\bf u}, {\bf w}} \in \Omega$  and $X \subset  \ocs({\bf u})$, the set $X$ is $f_{{\bf u}, {\bf v}}$-stable in the identity ${\bf u} \approx {\bf v}$ and the set $f_{{\bf u}, {\bf v}} (X)$ is $f_{{\bf v}, {\bf w}}$-stable in the identity ${\bf v} \approx {\bf w}$ then the set $X$ is $f_{{\bf u}, {\bf w}}$-stable in the identity ${\bf u} \approx {\bf w}$.
\end{fact}

\begin{proof}
Since the set $X$ is $f_{{\bf u}, {\bf v}}$-stable in the identity ${\bf u} \approx {\bf v}$ and the set $f_{{\bf u}, {\bf v}} (X)$ is $f_{{\bf v}, {\bf w}}$-stable in the identity ${\bf v} \approx {\bf w}$, the map $f_{{\bf u}, {\bf w}}$ is defined on $X$ and by Definition \ref{trans}(iii) we have $(f_{{\bf u}, {\bf v}} \circ f_{{\bf v}, {\bf w}}) (X) = f_{{\bf u}, {\bf w}} (X)$. Since $f_{{\bf u}, {\bf w}}$ is a composition of two isomorphisms, it is an isomorphism from $X$ to $f_{{\bf u}, {\bf w}} (X)$. Consequently, the set $X$ is $f_{{\bf u}, {\bf w}}$-stable in the identity ${\bf u} \approx {\bf w}$.
\end{proof}

 A {\em derivation} of an identity ${\bf U} \approx {\bf V}$ from $\Sigma$ is a sequence of words ${\bf U}={\bf U}_1 \approx {\bf U}_2 \approx \dots \approx {\bf U}_l={\bf V}$ and substitutions $\Theta_1, \dots, \Theta_{l-1} (\mathfrak A \rightarrow \mathfrak A ^+$) such that for each  $i=1, \dots, l-1$ we have ${\bf U}_i=\Theta_i({\bf u}_i)$ and  ${\bf U}_{i+1}=\Theta_i({\bf v}_i)$ for some identity ${\bf u}_i \approx {\bf v}_i \in \Sigma$. According to \cite[Section 9]{OSPV}, each set of identities $\Sigma$ in at most $n$ variables is a subset of a certain set of
 identities $\Sigma'$ in at most $n+2$ variables such that an identity $\tau$ can be derived from $\Sigma$ in the usual sense if and only if $\tau$ can be derived from $\Sigma'$ in the sense defined in the previous sentence.

 We say that a variable $x$ is {\em stable} in an identity  ${\bf u} \approx {\bf v}$ if $occ_{\bf u}(x) = occ_{\bf v}(x)$. Otherwise, we say that $x$ is {\em unstable} in ${\bf u} \approx {\bf v}$. The following statement is a specialization of Fact \ref{nfblemma}.

\begin{lemma} \label{nfblemma1} Let $S$ be a semigroup.

Suppose that there exists a good collection of maps $\{ f_{{\bf w}, {\bf w'}} \mid ({\bf w} \approx {\bf w'}) \in \mathfrak A^+ \times \mathfrak A^+\}$ such that for each $n$ large enough one can find an identity ${\bf U}_n \approx {\bf V}_n$ of $S$ in at least $n$ variables
and some possibly empty sets $X_1 \subseteq \ocs({\bf U}_n), \dots, X_k \subseteq \ocs({\bf U}_n)$ and $\mathfrak X \subseteq \con({\bf U}_n)$ such that each of the following conditions is satisfied:

(I) the map $f_{{\bf U}_n, {\bf U}_n}$ is defined on $X_1 \cup \dots \cup X_k$;

(II) either for some $1 \le i \le k$ the set $X=X_i$ is $f_{{\bf U}_n, {\bf V}_n}$-unstable in ${\bf U}_n \approx {\bf V}_n$ or some variable $x \in \mathfrak X$ is unstable in ${\bf U}_n \approx {\bf V}_n$;

(III) if ${\bf U} \in [\![{\bf U}_n]\!]_S$, each variable $x \in \mathfrak X$ is stable in ${\bf U}_n \approx {\bf U}$ and all sets $X_1, \dots, X_k$ are $f_{{\bf U}_n, {\bf U}}$-stable in ${\bf U}_n \approx {\bf U}$ then for every identity ${\bf u} \approx {\bf v}$ of $S$ in less than $n/4$ variables and every substitution  $\Theta: \mathfrak A \rightarrow \mathfrak A^+$ such that $\Theta({\bf u})={\bf U}$, each variable $x \in \mathfrak X$ is stable in ${\bf U} \approx \Theta({\bf v})$
and all sets $f_{{\bf U}_n, {\bf U}}(X_1), \dots f_{{\bf U}_n, {\bf U}}(X_k)$ are $f_{{\bf U}, \Theta({\bf v})}$-stable in ${\bf U} \approx \Theta({\bf v})$.

Then the semigroup $S$ is non-finitely based.
\end{lemma}

\begin{proof} Take $m>0$ and let $\Sigma$ be a set of identities of $S$ in at most $m$ variables.
By our assumption, there exists a good collection of maps $\Omega = \{ f_{{\bf w}, {\bf v}} \mid ({\bf w}, {\bf v}) \in \mathfrak A^+ \times \mathfrak A^+\}$ such that
one can find an identity ${\bf U}_{4m+8} \approx {\bf V}_{4m+8}$ of $S$ in at least $4m+8$ variables and some sets $X_1 \subseteq \ocs({\bf U}_{4m+8}), \dots, X_k \subseteq \ocs({\bf U}_{4m+8})$ and $\mathfrak X \subseteq \con({\bf U}_{4m+8})$ such that Conditions (I) -- (III) are satisfied.

If the identity ${\bf U}_{4m+8} \approx {\bf V}_{4m+8}$ is a consequence from $\Sigma$ then one can find a set of identities $\Sigma'$ in at most $m+2$ variables, a sequence of words ${\bf U}_{4m+8}={\bf W}_1 \approx {\bf W}_2 \approx \dots \approx {\bf W}_l={\bf V}_{4m+8}$ and substitutions $\Theta_1, \dots, \Theta_{l-1} (\mathfrak A \rightarrow \mathfrak A ^+$) such that for each  $i=1, \dots, l-1$ we have ${\bf W}_i=\Theta_i({\bf u}_i)$ and  ${\bf W}_{i+1}=\Theta_i({\bf v}_i)$ for some identity ${\bf u}_i \approx {\bf v}_i \in \Sigma'$.

Since by Condition (I) the map $f_{{\bf W}_1, {\bf W}_1}$ is defined on $X_1 \cup \dots \cup X_k$, it is the identity map on $X_1 \cup \dots \cup X_k$
in view of Definition \ref{trans} (ii). Therefore, for each $1 \le i \le k$,
the set $X_i$ is $f_{{\bf W_1}, {\bf W_1}}$-stable in ${\bf W}_1 \approx {\bf W}_1$. Now Condition (III) implies that the set $X_i$ is $f_{{\bf W}_1, {\bf W}_2}$-stable in ${\bf W}_1 \approx {\bf W}_2$ and the set $f_{{\bf W}_1, {\bf W}_2}(X_i)$ is $f_{{\bf W}_2, {\bf W}_3}$-stable in ${\bf W}_{2} \approx {\bf W}_3$. Therefore, by Fact \ref{gcm}, the set $X_i$ is $f_{{\bf W}_1, {\bf W}_3}$-stable in ${\bf W}_1 \approx {\bf W}_3$. And so on. Eventually, we obtain that the set $X_i$ is $f_{{\bf W}_1, {\bf W}_l}$-stable in ${\bf U}_{4m+8}={\bf W}_1 \approx {\bf W}_l = {\bf V}_{4m+8}$ for each $1 \le i \le k$.

Condition (III) also implies that each variable in $\mathfrak X$ is stable in ${\bf U}_{4m+8} \approx {\bf V}_{4m+8}$.
To avoid a contradiction with Condition (II) we must conclude that the identity ${\bf U}_{4m+8} \approx {\bf V}_{4m+8}$ is not a consequence from $\Sigma$.
Since $m$ and $\Sigma$ were arbitrary, $S$ is non-finitely based.
\end{proof}

If ${\bf u}$ and ${\bf v}$ are two words with $\con({\bf u}) \cap \con({\bf v}) \ne \emptyset$ then $l_{{\bf u}, {\bf v}}$ is a map from $\{ _{i{\bf u}}x \mid x \in \con({\bf u}), i \le \min(occ_{\bf u} (x),occ_{\bf v} (x))\}$ to $\{ _{i{\bf v}}x \mid x \in \con({\bf v}), i \le \min(occ_{\bf u} (x),occ_{\bf v} (x))\}$ defined by $l_{{\bf u}, {\bf v}} (_{i{\bf u}}x) = {_{i{\bf v}}x}$. For example, if ${\bf u} = zxyzyyx$ and ${\bf v} = xxyxpp$ then $l_{{\bf u}, {\bf v}}$ is a bijection between $\{_{1{\bf u}}x, {_{1{\bf u}}y}, {_{2{\bf u}}x} \}$ and $\{_{1{\bf v}}x, {_{1{\bf v}}y}, {_{2{\bf v}}x} \}$. The set $\{_{1{\bf u}}x, {_{1{\bf u}}y}, {_{2{\bf u}}x} \}$
is $l_{{\bf u}, {\bf v}}$ - unstable in ${\bf u} \approx {\bf v}$ but the set $\{_{1{\bf u}}x, {_{1{\bf u}}y} \}$
is $l_{{\bf u}, {\bf v}}$ - stable in ${\bf u} \approx {\bf v}$.

It is easy to check that the set $\{ l_{{\bf u}, {\bf v}} \mid ({\bf u} \approx {\bf v}) \in \mathfrak A^+ \times \mathfrak A^+, \con({\bf u}) \cap \con({\bf v}) \ne \emptyset \}$ satisfies all three conditions of Definition \ref{trans}, and consequently is a good collection of maps.

An identity ${\bf u} \approx {\bf v}$ is called {\em regular} if $\con({\bf u}) = \con({\bf v})$. If ${\bf u} \approx {\bf v}$ is a regular identity then evidently,
${\bf u} \approx {\bf v}$ is trivial if and only if the set $\ocs({\bf u})$ is $l_{{\bf u}, {\bf v}}$ - stable in ${\bf u} \approx {\bf v}$ and each variable $x \in \con({\bf u})$ is stable in ${\bf u} \approx {\bf v}$.

Notice that under an additional assumption that $S$ satisfies only regular identities, Corollary \ref{nfbcor} can be obtained readily from  Lemma \ref{nfblemma1}.
(Simply choose $\{ l_{{\bf u}, {\bf v}} \mid ({\bf u} \approx {\bf v}) \in \mathfrak A^+ \times \mathfrak A^+, \con({\bf u}) \cap \con({\bf v}) \ne \emptyset \}$
for the role of the good collection of maps, $k=1$, $X=X_1 = \ocs({\bf U}_n)$ and $\mathfrak X = \con({\bf U}_n)$.)

The rest of this article is devoted to twelve applications of Lemma \ref{nfblemma1}.
Almost every time we use Lemma \ref{nfblemma1} (except for Theorem \ref{PSC}) we will choose $k=1$.
In addition, almost every time we use Lemma \ref{nfblemma1} (except for Theorem \ref{EL}) we will choose $\{ l_{{\bf u}, {\bf v}} \mid ({\bf u} \approx {\bf v}) \in \mathfrak A^+ \times \mathfrak A^+, \con({\bf u}) \cap \con({\bf v}) \ne \emptyset \}$ for the role of the good collection of maps. Since $l_{{\bf u}, {\bf u}}$ is the identity map of $\ocs({\bf u})$ for each ${\bf u} \in \mathfrak A^+$, we will omit checking Condition (I) in Lemma \ref{nfblemma1}.

Let $\Theta: \mathfrak A \rightarrow \mathfrak A^+$ be a substitution such that $\Theta({\bf u})={\bf U}$. Then
$\Theta$ induces a map $\Theta_{\bf u}$ from $\ocs({\bf u})$
into  subsets of $\ocs({\bf U})$ as follows.
If $1 \le i \le occ_{\bf u}(x)$ then $\Theta_{\bf u}({_{i{\bf u}}x})$ denotes the set of all elements of $\ocs({\bf U})$
contained in the subword of ${\bf U}$ of the form $\Theta(x)$ that corresponds to the $i^{th}$ occurrence of variable $x$ in ${\bf u}$. For example, if $\Theta(x)=ab$ and $\Theta(y)=bab$ then $\Theta_{xyx}({_{2(xyx)}x})=\{{_{3(abbabab)}a}, {_{4(abbabab)}b} \}$.
Evidently, for each $d \in \ocs({\bf u})$ the set $\Theta_{\bf u} (d)$ is an interval in $(\ocs({\bf U}), <_{\bf U})$.
Now we define a function $\Theta^{-1}_{\bf u}$ from $\ocs({\bf U})$ to $\ocs({\bf u})$ as follows.
If $c \in \ocs({\bf U})$ then $\Theta^{-1}_{\bf u}(c)=d$ when $\Theta_{\bf u} (d)$ contains $c$.
For example, $\Theta^{-1}_{xyx}({_{3(abbabab)}a})= {_{2(xyx)}x}$.
It is easy to see that if ${\bf U}=\Theta({\bf u})$ then the function $\Theta^{-1}_{\bf u}$ is a homomorphism from $(\ocs({\bf U}), <_{\bf U})$ to $(\ocs({\bf u}), <_{\bf u})$, i.e. \ for every $c,d \in \ocs({\bf U})$ we have $\Theta^{-1} _{\bf u}(c) \le _{\bf u} \Theta^{-1}_{\bf u}(d)$  whenever $c <_{\bf U} d$.
The following fact can be easily verified.

\begin{fact} \label{thetainv} Let ${\bf u}$ be a word and $\Theta: \mathfrak A \rightarrow \mathfrak A^+$ be a substitution such that $\Theta({\bf u})={\bf U}$.

   If  $\Theta^{-1} _{\bf u}({_{i{\bf U}}x})=({_{j{\bf u}}y})$ for some $x \in \con({\bf U})$ and $y \in \con({\bf u})$ then $occ_{\bf u}(y) \le occ_{\bf U}(x)$
   and $j \le i$.

\end{fact}

 If $\mathfrak X$ is a set of variables then we write ${\bf u}(\mathfrak X)$ to refer to the word obtained from ${\bf u}$ by deleting all occurrences of all variables that are not in $\mathfrak X$ and say that the word ${\bf u}$ {\em deletes} to the word ${\bf u}(\mathfrak X)$. If $\mathfrak X = \{x_1, \dots, x_n\}$ then we write
 ${\bf u}(x_1, \dots, x_n)$ instead of ${\bf u}(\{x_1, \dots, x_n\})$.
 For simplicity, we sometimes write  ${_ix} <_{\bf U} {_jy}$  instead of ${_{i{\bf U}}x} <_{\bf U} {_{j{\bf U}}y}$.
 If a variable $t$ is {\em linear} (1-occurring) in a word ${\bf u}$ then we use ${_{\bf u}t}$ to denote
 the only occurrence of $t$ in ${\bf u}$. We use $\lin ({\bf u})$ to denote the set of all linear variables in a word $\bf u$ and $\non ({\bf u})$ to denote the set of all non-linear variables in $\bf u$.

  We say that a set $X \subseteq \ocs({\bf u})$ is {$f_{{\bf u}, {\bf v}}$-stable in ${\bf u}$  with respect to a semigroup $S$} if $X$ is $f_{{\bf u}, {\bf v}}$-stable in any identity of $S$ of the form ${\bf u} \approx {\bf v}$.

\begin{theorem} \label{SL1} Let $S$ be a monoid that satisfies the following conditions:

(i) For each $n$ large enough, $S$ satisfies the identity ${\bf U}_n= [Xn][Yn][nX][nY] \approx [Yn][Xn][nY][nX]={\bf V}_n$;

(ii) For each $m+c>0$, the set $\{{_1x}, t\}$ is $l_{{\bf u}, {\bf v}}$-stable in $x^mtx^c$ with respect to $S$;

(iii) For each $m,c>1$ the equivalence class $[\![x^my^c]\!]_S$ contains only words of the form $x^iy^j$ for some $i,j >1$;

(iv) For each $m,c>0$ and $d>1$ the equivalence class $[\![x^my^dx^c]\!]_S$ contains only words of the form $x^iy^jx^k$ for some $i,k >0$ and $j>1$.

  Then the monoid $S$ is non-finitely based.

\end{theorem}

\begin{proof}
First, notice that if ${\bf u} \approx {\bf v}$ is an identity of $S$ then Conditions (iii) and (iv) imply that $\lin ({\bf u}) = \lin ({\bf v})$ and $\non ({\bf u}) = \non ({\bf v})$. In particular, $S$ satisfies only regular identities.
We need the following property of the equivalence class $[\![{\bf U}_n]\!]_S$.

\begin{claim} \label{L} Let $n>2$, $1< i < n$ and ${\bf U} \in [\![{\bf U}_n]\!]_S$.
If ${_{1{\bf U}}x}_i <_{\bf U} {_{1{\bf U}}y}_2$  then ${_{1{\bf U}}x}_{i+1} <_{\bf U} {_{1{\bf U}}y}_2$.
\end{claim}

\begin{proof}
Working toward a contradiction, assume that for some $i>1$ we have ${_1x}_i <_{\bf U} {_1y}_2$  but ${_1y}_2 <_{\bf U} {_1x}_{i+1}$.
Then $({_1x}_i) <_{\bf U} ({_1y}_2) <_{\bf U} ({_1x}_{i+1}) <_{\bf U} ({_{last} x}_{i+1})$.

Where is ${_{last} y}_2$? Since ${\bf U}_n (y_2, x_{i+1}) = (x_{i+1}) (y_2) (x_{i+1}) (y_2)$, in order to avoid a contradiction to Condition (iii), we
conclude that $({_1x}_{i+1}) <_{\bf U}{_{last} y}_2$.

Now assume that  ${_{last} x}_{i+1} <_{\bf U} {_{last} y}_2$.
Since ${\bf U}_n (y_1, y_2) = y_1 y_2 y_2 y_1$, Condition (iv) implies that ${\bf U} (y_1, y_2) = y_1^m y_2^d y_1^c$ for some $m,c>0$ and $d>1$.
So, we have $({_my}_1) <_{\bf U} ({_1y}_2) <_{\bf U} ({_1 x}_{i+1}) <_{\bf U} ({_{last} x}_{i+1}) <_{\bf U} ({_{last} y}_2) <_{\bf U} ({_{m+1} y}_1)$.
This implies that ${\bf U}( y_1, x_{i+1}) = y_1 ^m x_{i+1}^d y_1^c$ for some $m,c>0$ and $d>1$.
Now since ${\bf U}_n (y_1, x_{i+1}) = (x_{i+1}) (y_1) (x_{i+1}) (y_1)$, to avoid a contradiction to Condition (iv) we conclude that ${_{last} y}_2 <_{\bf U}  {_{last} x}_{i+1}$. So, we have that $({_1x}_i) <_{\bf U} ({_1y}_2) <_{\bf U} ({_1x}_{i+1})<_{\bf U} ({_{last} y}_2) <_{\bf U}({_{last} x}_{i+1})$.

 Since ${\bf U}_n (x_i, x_{i+1}) = x_i x_{i+1} x_{i+1} x_i$ to avoid a contradiction to Condition (iv), we must assume that
$({_1x}_i) <_{\bf U} ({_1y}_2) <_{\bf U} ({_1x}_{i+1})<_{\bf U} ({_{last} y}_2) <_{\bf U} ({_{last} x}_{i+1}) <_{\bf U} ({_{last} x}_{i})$. Now since
${\bf U}_n (x_{i-1}, x_{i}) = x_{i-1} x_{i} x_{i} x_{i-1}$, to avoid a contradiction to Condition (iv), we must assume that
${\bf U} (x_{i-1}, y_{2}) = x_{i-1} ^m y_2^d x_{i-1}^c$ for some $m,c>0$ and $d>1$. Now since ${\bf U}_n (x_{i-1}, y_2) = (x_{i-1}) (y_2) (x_{i-1}) (y_2)$,
this gives us a contradiction to Condition (iv) that we can do nothing to avoid. \end{proof}

Now take $n>4$ and consider the identity ${\bf U}_n \approx {\bf V}_n$ and set $X=\{{_{1{\bf U}_n}x}_2, {_{1{\bf U}_n}y}_2\}$.
Evidently, the set $X$ is $l_{{\bf U}_n, {\bf V}_n}$-unstable in ${\bf U}_n \approx {\bf V}_n$. We choose $\mathfrak X$ to be the empty set of variables.

Let us check the third condition of Lemma \ref{nfblemma1}. Let ${\bf U} \in [\![{\bf U}_n]\!]_S$ be a word such that the set $X=\{{_{1{\bf U}_n}x}_2, {_{1{\bf U}_n}y}_2\}$ is $l_{{\bf U}_n, {\bf U}}$-stable in ${\bf U}_n \approx {\bf U}$, i.e. ${_{1{\bf U}}x}_2 <_{\bf U} {_{1{\bf U}}y}_2$.

 Since for each $1<i<n-1$ we have ${\bf U}_n (x_i, x_{i+1})= x_ix_{i+1}^2x_i$,
Condition (iv) implies the following property of the word ${\bf U}$.

(P)  For each $1<i<n-1$ we have ${\bf U}(x_i, x_{i+1}) = x_i^mx_{i+1}^px_i^{q}$ for some $m, q >0$ and $p>1$.

  Property (P) and Claim \ref{L} imply that
$({_{1{\bf U}}x}_2) <_{\bf U} ({_{1{\bf U}}x}_3) <_{\bf U} \dots <_{\bf U} ({_{1{\bf U}}x}_n) <_{\bf U} ({_{1{\bf U}}y}_2)$.

Let ${\bf u}$ be a word in less than $n/2$ variables so that $\Theta({\bf u})={\bf U}$ for some substitution $\Theta: \mathfrak A \rightarrow \mathfrak A^+$.
Since the word ${\bf u}$ has less than $n/2$ variables, for some
$c \in \ocs({\bf u})$ and $2 < i < n-1$ both ${_{1{\bf U}}x}_i$ and ${_{1{\bf U}}x}_{i+1}$ are contained in $\Theta_{\bf u}(c)$.
Then property (P) implies that $c$ must be the only occurrence of a linear variable $t$ in ${\bf u}$.

Since $\Theta^{-1}_{\bf u}$ is a homomorphism from $(\ocs({\bf U}), <_{\bf U})$ to $(\ocs({\bf u}), <_{\bf u})$,
we have that $\Theta^{-1}_{\bf u}({_{1{\bf U}}x}_2) \le_{\bf u} ({_{\bf u}t})  \le_{\bf u} \Theta^{-1}_{\bf u}({_{1{\bf U}}y}_2)$.

Now let ${\bf u} \approx {\bf v}$ be an arbitrary identity of $S$ and ${\bf V} = \Theta({\bf v})$.
In view of Fact \ref{thetainv}, for some $x,x' \in \con({\bf u})$ and $y, y' \in \con({\bf v})$ we have $_{1{\bf u}}x =\Theta^{-1}_{\bf u}({_{1{\bf U}}x}_2)$, $_{1{\bf u}}y = \Theta^{-1}_{\bf u}({_{1{\bf U}}y}_2)$, $_{1{\bf v}}x' =\Theta^{-1}_{\bf v}({_{1{\bf V}}x}_2)$ and $_{1{\bf v}}y' = \Theta^{-1}_{\bf v}({_{1{\bf V}}y}_2)$.
 Then we have $(_{1{\bf u}}x) \le_{\bf u} ({_{\bf u}t})  \le_{\bf u} ({ _{1{\bf u}} y}) \le_{\bf u} ({ _{1{\bf u}} y'})$.

Now in view of Condition (ii) we have
that $({_{1{\bf v}}x'}) \le_{\bf v} ({_{1{\bf v}}x}) \le_{\bf v} ({_{\bf v}t})  \le_{\bf v} ({ _{1{\bf v}} y})$ and $({_{\bf v}t})  \le_{\bf v} ({ _{1{\bf v}} y'})$.

Therefore, $\Theta^{-1}_{\bf v}({_{1{\bf V}}x}_2) = ({_{1{\bf v}}x'}) \le_{\bf v} ({_{\bf v}t})  \le_{\bf v} ({_{1{\bf v}}y'}) = \Theta^{-1}_{\bf v}({_{1{\bf V}}y}_2)$ and, consequently, $({_{1{\bf V}}x}_2) <_{\bf V} ({_{1{\bf V}}y}_2)
$. This means that the set $l_{{\bf U}_n, {\bf U}}(X) = \{{_{1{\bf U}}x}_2, {_{1{\bf U}}y}_2 \}$ is $l_{{\bf U}, {\bf V}}$-stable in ${\bf U} \approx {\bf V}$. Therefore, the monoid $S$ is non-finitely based  by Lemma \ref{nfblemma1}.
\end{proof}

\begin{cor} \label{WT} \cite[Theorem 11]{WTZ} Let $L^1$ denote the monoid obtained by adjoining an identity element to the semigroup $L= \langle a,b \mid aa=a, bb=b, aba=0 \rangle$ of order six. Then $L^1$ is non-finitely based.
\end{cor}

\begin{proof} (i) According to Lemma 10 in \cite{WTZ}, the monoid $L^1$ satisfies ${\bf U}_n \approx {\bf V}_n$ for each $n \ge 2$.

(ii) First notice that since $(ab)^2 = 0$ and $ab \ne 0$, the word $t$ is an isoterm for $L^1$. Now Condition (ii)
can be checked easily by substituting $x \rightarrow b$ and $t \rightarrow a$.

Conditions (iii)--(iv) can be checked easily by substituting $x \rightarrow b$ and $y \rightarrow a$. So, by Theorem \ref{SL1}
the monoid $L^1$ is non-finitely based.
\end{proof}

\section{Interrelations between words induced by monoids}

If $W$ and $W'$ are two sets of words then we write $W \preceq W'$ if for any monoid $S$ each word in $W'$ is an isoterm for $S$ whenever  each word in $W$ is an isoterm for $S$. It is easy to see that the relation $\preceq$ is reflexive and transitive, i.e. it is a {\em quasi-order} on sets of words. If $W \preceq W' \preceq W$ then we write $W \sim W'$. The relations $\preceq$ and $\sim$ can be extended to individual words. For example, if $\bf u$ and $\bf v$ are two words then ${\bf u} \sim {\bf v}$ means $\{{\bf u}\} \sim \{{\bf v}\}$. Also, if $W$ is a set of words and $\bf u$ is a word then $W \preceq {\bf u}$  means $W \preceq \{{\bf u}\}$.

For example, it is easy to see that for every $n>1$ we have $xy \preceq x_1x_2 \dots x_n$.
If ${\bf u}$ is a subword of a word ${\bf v}$ then evidently, we have ${\bf v} \preceq {\bf u}$. The following fact is an immediate consequence
of known bases for $S(\{ata\})$ (see Lemma 4.5 in \cite{MJ}, for instance) but we prove it directly.

\begin{fact} \label{xtx}
If $xtx$ is an isoterm for a monoid $S$, then

(i) the words $xt_1yxt_2y$ and $xt_1xyt_2y$ can only
form an identity of $S$ with each other;

(ii) the words $xyt_1xt_2y$ and $yxt_1xt_2y$ can only
form an identity of $S$ with each other;

(iii) the words $xt_1yt_2xy$ and $xt_1yt_2yx$ can only
form an identity of $S$ with each other.
\end{fact}

\begin{proof} (i) If $S$ satisfies an identity $xt_1yxt_2y \approx {\bf u}$ then since $xtx \preceq xy$ we have ${\bf u}(t_1, t_2)=t_1t_2$.
Since $xtx$ is an isoterm for $S$ we have that ${\bf u}(t_1, t_2, x) = xt_1xt_2$ and ${\bf u}(t_1, t_2, y) = t_1yt_2y$. Therefore, the words $xt_1yxt_2y$ and $xt_1xyt_2y$ can only form an identity of $S$ with each other.

The proofs of Parts (ii) and (iii) are similar. \end{proof}

Since each of the six words considered in Fact \ref{xtx} is less than the word $xtx$ in the order $\preceq$, Fact \ref{xtx} immediately implies the following statement that will be often used without reference.

\begin{fact} \label{xy3}

$xt_1xyt_2y \sim xt_1yxt_2y$,  $xyt_1xt_2y \sim yxt_1xt_2y$ and  $xt_1yt_2xy \sim xt_1yt_2yx$.

\end{fact}

We say that a set of variables $\mathfrak X$ is {\em stable} in an identity $\bf u \approx \bf v$ if ${\bf u}(\mathfrak X)={\bf v}(\mathfrak X)$. Otherwise, we say that set $\mathfrak X$ is {\em unstable} in $\bf u \approx \bf v$.  We say that a set of variables $\mathfrak X$ is {\em stable} in a word $\bf u$ with respect to $S$ if set $\mathfrak X$ is stable in every identity of $S$ of the form $\bf u \approx \bf v$.
The following fact can be easily verified.

\begin{fact} \label{linst} Let $S$ be a monoid and ${\bf u}$ be a word.

 (i) If the word $x$ is an isoterm for $S$ then $S$ satisfies only regular identities.

 (ii) If the  word $x^m$ is an isoterm for $S$ then every variable $x$ with $occ_{\bf u}(x) \le m$ is stable in ${\bf u}$  with respect to $S$.

(iii) If the word $xy$ is an isoterm for $S$ then the set of all linear variables in  ${\bf u}$ is stable in ${\bf u}$ with respect to $S$.
\end{fact}

If $X \subseteq \ocs ({\bf u})$ then $\con (X)$ denotes the set of all variables involved in $X$.
We say that a set $X \subseteq \ocs ({\bf u})$ is {\em stable} in ${\bf u}$ with respect to $S$ if $X$ is left-stable in ${\bf u}$ with respect to $S$ and each variable in $\con (X)$ is stable in ${\bf u}$ with respect to $S$. If $\{c,d\} \subseteq X \subseteq \ocs ({\bf u})$, then we say that the pair $\{c,d\}$ is {\em adjacent in $X$} if there is no element $e \in X$
such that $c <_{\bf u} e <_{\bf u} d$. If $X= \ocs ({\bf u})$ then we simply say that the pair $\{c,d\}$ is {\em adjacent in ${\bf u}$}.
We say that a pair of variables $\{x,y\}$ is {\em adjacent in ${\bf u}$} if for some $c,d \in \ocs ({\bf u}(x,y))$ the pair $\{c,d\}$ is adjacent
in ${\bf u}$. The following two facts can be easily verified and will be sometimes used without reference.

\begin{fact} \label{st1} Let $S$ be a monoid and ${\bf u}$ be a word.
The following conditions are equivalent:

(i) ${\bf u}$ is an isoterm for $S$;

(ii) The set $\con ({\bf u})$ is stable in ${\bf u}$ with respect to $S$;

(iii) The set $\ocs ({\bf u})$ is stable in ${\bf u}$ with respect to $S$;

(iv) Each adjacent pair in $\ocs ({\bf u})$ is stable in ${\bf u}$ with respect to $S$;

(v) Each adjacent pair in $\con ({\bf u})$ is stable in ${\bf u}$ with respect to $S$.

\end{fact}

\begin{fact} Let ${\bf u} \approx {\bf v}$ be an identity.

(i) If a set $X \subseteq \ocs ({\bf u})$ is stable in ${\bf u} \approx {\bf v}$, then
every subset of $X$ is also stable in ${\bf u} \approx {\bf v}$.

(ii) If a set $\mathfrak X \subseteq \con ({\bf u})$ is stable in ${\bf u} \approx {\bf v}$, then
every subset of $\mathfrak X$ is also stable in ${\bf u} \approx {\bf v}$.

(iii) If a set $\mathfrak X \subseteq \con ({\bf u})$ is stable in ${\bf u} \approx {\bf v}$, then
 $\ocs ({\bf u}(\mathfrak X))$ is also stable in ${\bf u} \approx {\bf v}$.

\end{fact}

 We use letter $t$ with or without subscripts to denote linear (1-occurring) variables. If we use letter $t$ several times in a word,  we assume that different occurrences of $t$ represent distinct linear variables. The word $x_1 y_1 x_2 y_2 \dots x_{n}y_n$ is denoted by $[XYn]$. We use ${\bf U}^t$ ($^t{\bf U}$) to denote the word obtained from a word ${\bf U}$ by inserting a linear letter after (before) each occurrence of each variable in ${\bf U}$. For example, $[Zn]^t = z_1t_1z_2t_2 \dots z_nt_n$.

The next three lemmas are needed only to prove Theorem \ref{nfbpairs} which is associated with Table \ref{classes}. We start with a statement related to the first row in Table \ref{classes}.

\begin{lemma} \label{row1} Let $S$ be a monoid such that the word $xyyx$ is an isoterm for $S$.
Let ${\bf u}$ be a word that satisfies the following conditions:

 (i) some variables $x, y \in \con({\bf u})$ occur at most twice in ${\bf u}$;

 (ii) $({_{1{\bf u}}x}) <_{\bf u} ({ _{1{\bf u}} y})$;

(iii) If both variables $x$ and $y$ occur twice in ${\bf u}$ then ${\bf u}$ contains a linear variable $t$ such that
$({_{1{\bf u}}y}) <_{\bf u} ({_{\bf u}t}) <_{\bf u} ({ _{2{\bf u}} y})$  and $({_{2{\bf u}}x}) <_{\bf u} ({_{\bf u}t}) <_{\bf u} ({ _{2{\bf u}} y})$.

Then the set $\{{_{1{\bf u}}x}, {_{1{\bf u}}y} \}$ is stable in ${\bf u}$ with respect to $S$.

\end{lemma}

\begin{proof} Let ${\bf u} \approx {\bf v}$ be an identity of $S$.
Since the word $xx$ is an isoterm for $S$, we have that $occ_{\bf u}(x)= occ_{\bf v}(x)$ and
$occ_{\bf u}(y)= occ_{\bf v}(y)$.

If both $x$ and $y$ are linear in ${\bf u}$ then the set $\{{_{1{\bf u}}x}, {_{1{\bf u}}y} \}$ is stable in ${\bf u} \approx {\bf v}$ by Fact \ref{linst}.
If one variable among $x$ and $y$ is linear in ${\bf u}$, then the set $\{{_{1{\bf u}}x}, {_{1{\bf u}}y} \}$ is stable in ${\bf u} \approx {\bf v}$
because $xyyx \preceq \{xtx, xxt, txx\}$.

If both $x$ and $y$ occur twice in ${\bf u}$ then either  ${\bf u}(x, y, t) = xxyty$ or ${\bf u}(x, y, t) = xyxty$.
In both cases, the set $\{{_{1{\bf u}}x}, {_{1{\bf u}}y} \}$ is stable in ${\bf u} \approx {\bf v}$.
\end{proof}

The next lemma is related to the second row in Table \ref{classes}.

\begin{lemma} \label{row2}

 $\{yxxty, ytxxy\} \preceq z_1 t_1 z_2 t_2 \dots z_n t_n xx z_1 z_2 \dots z_n$.

\end{lemma}

\begin{proof} Let $S$ be a monoid such that the words $yxxty$ and $ytxxy$ are isoterms for $S$.
Denote ${\bf u}= z_1 t z_2 t \dots z_n t xx z_1 z_2 \dots z_n$.

Since $xx$ is an isoterm for $S$ and every variable occurs in ${\bf u}$ at most twice, every variable in $\con({\bf u})$ is stable in ${\bf u}$ with respect to $S$.
Since $yxxty \preceq \{xxt, txx, xtx \}$, each pair of variables that involves a linear variable  is stable in ${\bf u}$ with respect to $S$.
Since for each  $1 \le i< j \le n$ we have ${\bf u}(z_i, z_j, t_i, t_n)= z_it_i z_j t_n z_i z_j \sim ytxtyx \sim ytxtxy \succeq  ytxxy$, for each  $1 \le i < j \le n$ the pair $(z_i, z_j)$ is stable in ${\bf u}$ with respect to $S$.
Since for each  $1 \le i \le n$ we have ${\bf u}(x, z_i, t_i)= z_i t_i xxz_i \sim ytxxy$, for each $1 \le i \le n$ the pair $(x, z_i)$ is stable in ${\bf u}$ with respect to $S$. The rest follows from Fact \ref{st1}.
\end{proof}

 The next lemma is related to the third row in Table \ref{classes}.

\begin{lemma} \label{row3}
 $xtxyty \preceq  [ZPn]^t \hskip .04in [ZQn] [PRn]\hskip .04in ^t[QRn] = $

 $=(z_1 t p_1 t z_2 t p_2 t \dots z_n t p_n t) (z_1 q_1 z_2 q_2 \dots z_n q_n) (p_1 r_1 p_2 r_2 \dots p_n r_n)(t q_1 t r_1 t q_2 t r_2 t \dots q_n t r_n).$
\end{lemma}

\begin{proof} Let $S$ be a monoid such that the word $xtxyty$ is an isoterm for $S$.
Denote ${\bf u}= [ZPn]^t\hskip .04in [ZQn] [PRn]\hskip .04in {^t[QRn]}$.

Notice that if $x$ is a non-linear variable in ${\bf u}$ then for some linear variable $t$ we have ${\bf u}(x,t)=xtx$.
This implies that each pair of variables that involves a linear variable is stable in ${\bf u}$ with respect to $S$.

Now for each $1 \le i \le n$ we have ${\bf u}(z_i,q_i)= z_i t z_i q_i t q_i$, ${\bf u}(q_i,z_{i+1})= z_{i+1} t q_i z_{i+1} t q_i$,
${\bf u}(q_n,p_{1})= p_{1} t q_n p_{1} t q_n$, ${\bf u}(p_i,r_i)= p_i t p_i r_i t r_i$, ${\bf u}(r_i,p_{i+1})= p_{i+1} t r_i p_{i+1} t r_i$.
In view of Fact \ref{xy3}, all these words are isoterms for $S$.
Since every adjacent pair of variables is stable in ${\bf u}$ with respect to $S$, the word ${\bf u}$ is an isoterm for $S$ by Fact \ref{st1}.
\end{proof}

As we mentioned in the introduction, Table \ref{classes} encodes eight sufficient conditions which are proved in Theorem \ref{nfbpairs} below.
The words $xyyx$ and $xtxyty$ in Rows 1 and 3 were used in \cite{MJ, JS, OS} as parts of different sufficient conditions (see more details at the end of Section 4). Also, as we mentioned in the introduction, the system of identities $\{xx[Yn][nY] \approx [Yn][nY]xx \mid n>1\}$ in Row 1 was used in \cite{Is, Sh}.

\begin{table}[tbh]
\begin{center}
\small
\begin{tabular}{|l|l|l|}
\hline    &   set of words $W$ &identity  ${\bf U}_n \approx {\bf V}_n$ for $n>1$  \\
\hline

\protect\rule{0pt}{10pt}  1 &  $xyyx$  & $xx[Yn][nY] \approx [Yn][nY]xx$\\
\hline

\protect\rule{0pt}{10pt}  2 &  $yxxty$, $ytxxy$&$[Zn]^t \hskip .04in yxx[Zn]y \approx [Zn]^t \hskip .04in xxy[Zn]y$ \\
\hline

\protect\rule{0pt}{10pt}  3 &   $xtxyty$& \parbox{10pt}{$[ZPn]^t \hskip .04in x[ZQn] xy [PRn]y \hskip .04in ^t[QRn] \approx [ZPn]^t \hskip .04in x[ZQn] yx [PRn]y \hskip .04in ^t[QRn]$}\\
\hline

\protect\rule{0pt}{10pt}  4 &  $xxyy$, $xytytx$&  $xytyz_1^2z_2^2 \dots z_n^2x \approx yxtyz_1^2z_2^2 \dots z_n^2x$ \\
\hline

\protect\rule{0pt}{10pt} 5 &  $xtyxty$, $xytxy$, $xytyx$ & $xy[Zn]yxt[nZ]  \approx yx[Zn]xyt[nZ]$\\
\hline

\protect\rule{0pt}{10pt} 6 &  $xtyxty$, $xytxy$, $xytyx$    & $xy[Zn]xyt[nZ] \approx yx[Zn]yxt[nZ]$\\
\hline

\protect\rule{0pt}{10pt}  7 &   $xtxyty$, $xyyx$    & $[Xn][nX][Yn][nY] \approx [Yn][nY][Xn][nX]$\\
\hline

\protect\rule{0pt}{10pt} 8 &  $xxyy$,& $yt_1x^{m-1}yp_1^2 \dots p_n^2zxt_2z \approx$
  \\

\protect\rule{0pt}{10pt} ($m>2$) &  $\{ytyx^dtx^{m-d}, x^{m-d}tx^dyty | 0 <d <m\}$ & $\approx yt_1x^{m}yp_1^2 \dots p_n^2zt_2z$\\
\hline

\end{tabular}
\caption{Eight sufficient conditions under which a monoid is NFB
\protect\rule{0pt}
{11pt}}

\label{classes}
\end{center}
\end{table}

\section{Eight sufficient conditions under which a monoid is non-finitely based}

If $\Theta: \mathfrak A \rightarrow \mathfrak A^+$ is a substitution and $\mathfrak Y$ is a set of variables then we define $\Theta^{-1}(\mathfrak Y):= \{x \in \mathfrak A \mid \con (\Theta(x)) \cap \mathfrak Y \ne \emptyset \}$. For example, if $\Theta(x)=abc$, $\Theta(y)=bab$ and $\Theta(z)=bb$  then $\Theta^{-1}(\{a,c \}) = \{x, y\}$.
If $y$ is a variable then we write $\Theta^{-1}(y)$ instead of $\Theta^{-1}(\{y\})$. (The set of variables $\Theta^{-1}(y)$ is
not to be confused with the function $\Theta^{-1}_{\bf u}$ which is a homomorphism from $(\ocs({\bf U}), <_{\bf U})$ to $(\ocs({\bf u}), <_{\bf u})$).

\begin{lemma} \label{stvar} Let ${\bf u} \approx {\bf v}$ and ${\bf U} \approx {\bf V}$ be two identities such that for some substitution  $\Theta: \mathfrak A \rightarrow \mathfrak A^+$ we have that $\Theta({\bf u})={\bf U}$ and $\Theta({\bf v})={\bf V}$.

Then a set of variables $\mathfrak Y$ is stable in ${\bf U} \approx {\bf V}$ whenever the set $\Theta^{-1}(\mathfrak Y)$ is stable in ${\bf u} \approx {\bf v}$.
\end{lemma}

\begin{proof} Denote $\mathfrak X=\Theta^{-1}(\mathfrak Y)$. Since $\mathfrak X$ is stable in ${\bf u} \approx {\bf v}$, we have that ${\bf u}(\mathfrak X) = {\bf v}(\mathfrak X)$.
Define a substitution $\Theta': \mathfrak A \rightarrow \mathfrak A^*$  by $\Theta'(x) = \Theta(x)(\mathfrak Y)$. Then ${\bf U}(\mathfrak Y)= \Theta'({\bf u}(\mathfrak X)) = \Theta'({\bf v}(\mathfrak X)) = {\bf V}(\mathfrak Y)$. Therefore, the set $\mathfrak Y$ is stable in ${\bf U} \approx {\bf V}$.
\end{proof}

The next statement can be easily verified.

\begin{fact} \label{2occ} Let ${\bf u}$ be a word and $\Theta: \mathfrak A \rightarrow \mathfrak A^+$ be a substitution.
Suppose that some variable $x$ occurs twice in ${\bf U} =\Theta({\bf u})$.

Then either the set $\Theta^{-1}(x)$ contains two variables $t_1$ and $t_2$ such that $occ_{\bf u}(t_1)= occ_{\bf u}(t_2)=1$, $\Theta^{-1}_{\bf u} ({_{1{\bf U}}x}) = {{_{\bf u}}t}_1$ and $\Theta^{-1}_{\bf u} ({_{2{\bf U}}x}) = {_{{\bf u}}t_2}$ or the set $\Theta^{-1}(x)$ contains a variable $x$ such that $occ_{\bf u}(x)=2$,
$\Theta^{-1}_{\bf u} ({_{1{\bf U}}x}) = {_{1{\bf u}}x}$ and $\Theta^{-1}_{\bf u} ({_{2{\bf U}}x}) = {_{2{\bf u}}x}$.

(The first possibility includes the case when $t_1 = t_2 =t$ which occurs when $\Theta(t)$ contains both occurrences of $x$ for some $t \in \con ({\bf u})$.)

\end{fact}

\begin{lemma} \label{samelabels}  Let $S$ be a monoid such that the word $xy$ is an isoterm for $S$. Suppose that $S$ satisfies an identity ${\bf u} \approx {\bf v}$ and there is a substitution $\Theta: \mathfrak A \rightarrow \mathfrak A^+$ and a variable $x$ such that $x$ appears twice in both ${\bf U}=\Theta({\bf u})$ and ${\bf V}=\Theta({\bf v})$. Then  $l_{{\bf u}, {\bf v}}(\Theta^{-1}_{\bf u} ({_{1{\bf U}}x}))=\Theta^{-1}_{\bf v} ({_{1{\bf V}}x})$ and $l_{{\bf u}, {\bf v}}(\Theta^{-1}_{\bf u} ({_{2{\bf U}}x}))=\Theta^{-1}_{\bf v} ({_{2{\bf V}}x})$.

\end{lemma}

\begin{proof} In view of Fact \ref{2occ}, there are only two possibilities for the set $\Theta^{-1}(x)$ and the word ${\bf u}$.

First, suppose that ${\bf u}$ contains two variables $t_1$ and $t_2$ so that  $occ_{\bf u}(t_1)= occ_{\bf u}(t_2)=1$, $\Theta^{-1}_{\bf u} ({_{1{\bf U}}x}) = {{_{\bf u}}t}_1$ and $\Theta^{-1}_{\bf u} ({_{2{\bf U}}x}) = {_{{\bf u}}t_2}$. Since the word $xy$ is an isoterm for $S$ this implies that ${\bf u}(t_1, t_2) = {\bf v}(t_1, t_2)$. Since both $\Theta(t_1)$ and $\Theta(t_2)$ contain $x$, the only possibility for word ${\bf v}$ is that  $\Theta^{-1}_{\bf v} ({_{1{\bf V}}x}) = {{_{\bf v}}t}_1$ and $\Theta^{-1}_{\bf v} ({_{2{\bf V}}x}) = {_{{\bf v}}t_2}$.

Now assume that ${\bf u}$ contains variable $x$ so that $occ_{\bf u}(x)=2$,
$\Theta^{-1}_{\bf u} ({_{1{\bf U}}x}) = {_{1{\bf u}}x}$ and $\Theta^{-1}_{\bf u} ({_{2{\bf U}}x}) = {_{2{\bf u}}x}$.
Since the word $x$ is an isoterm for $S$, we have $occ_{\bf v}(x)=2$.
Since $\Theta(x)$ contains $x$, we have
$\Theta^{-1}_{\bf v} ({_{1{\bf V}}x})= {_{1{\bf v}}x}$ and $\Theta^{-1}_{\bf v} ({_{2{\bf V}}x}) = {_{2{\bf v}}x}$.

 In any case we have $l_{{\bf u}, {\bf v}}(\Theta^{-1}_{\bf u} ({_{1{\bf U}}x}))=\Theta^{-1}_{\bf v} ({_{1{\bf V}}x})$ and $l_{{\bf u}, {\bf v}}(\Theta^{-1}_{\bf u} ({_{2{\bf U}}x}))=\Theta^{-1}_{\bf v} ({_{2{\bf V}}x})$. \end{proof}

The word obtained from a word ${\bf u}$ by deleting all occurrences of all variables in a set $\mathfrak X$ is denoted by $D_{\mathfrak X}({\bf u})$.
If the set $\mathfrak X$ contains only one variable $x$ then we simply write  $D_{x}({\bf u})$.

\begin{theorem} \label{nfbpairs}  Each row in Table \ref{classes} encodes a sufficient condition under which a monoid is non-finitely based.
 More precisely, each row in Table \ref{classes} encodes the following statement: if $S$ is a monoid that satisfies the identity ${\bf U}_n \approx {\bf V}_n$ for each $n>1$ and every word in the set $W$ is an isoterm for $S$ then $S$ is non-finitely based.
\end{theorem}

\begin{proof}
{\bf Row 1.} Here $W = \{xyyx\}$ and ${\bf U}_n = xx[Yn][nY] \approx [Yn][nY]xx ={\bf V}_n$. As we mentioned in the introduction, this sufficient condition can be
deduced easily either from the proof in \cite{Is} or from the proof in \cite{Sh}, but we are going to prove it by using Lemma \ref{nfblemma1}.
Let $S$ be a monoid such that the word $xyyx$ is an isoterm for $S$ and $S$ satisfies the identity ${\bf U}_n \approx {\bf V}_n$ for each $n>1$.

\begin{claim} \label{L1} If ${\bf U} \in [\![{\bf U}_n]\!]_S$ then ${\bf U}$ satisfies the following properties:

(P1) $occ_{\bf U}(x)=2$;

(P2) $D_x({\bf U})=[Yn][nY]=y_1y_2 \dots y_{n-1}y_ny_ny_{n-1} \dots y_2y_1$;

(P3) If ${_{1{\bf U}}x} <_{\bf U} {_{1{\bf U}}y}_1$ then ${_{2{\bf U}}x} <_{\bf U} {_{2{\bf U}}y}_n$.

\end{claim}

\begin{proof} Property (P1) follows from the fact that the word $xx$ is an isoterm for $S$ and $occ_{{\bf U}_n}(x)=2$.

Property (P2) follows from the fact that

 $D_x({\bf U}_n)=[Yn][nY]=y_1y_2 \dots y_{n-1}y_ny_ny_{n-1} \dots y_2y_1 \succeq xyyx$.

Let us verify Property (P3).
If ${_1 x} <_{\bf U} {_1y}_1$ but $ {_2y}_n <_{\bf U} {_2 x}$, then in view of Property (P2) we would have ${\bf U}(x,y_n)= xy_ny_nx  \ne  xxy_ny_n = {\bf U}_n(x, y_n)$. Since the word $xy_ny_nx$ is an isoterm for $S$, we must assume that ${\bf U}$ satisfies Property (P3).
\end{proof}

Now take $n>10$ and $X=\{{_{1{\bf U}_n}x}, {_{1{\bf U}_n}y}_1\}$. Evidently, the set $X=\{{_{1{\bf U}_n}x}, {_{1{\bf U}_n}y}_1\}$ is $l_{{\bf U}_n, {\bf V}_n}$-unstable in ${\bf U}_n \approx {\bf V}_n$.  We choose $\mathfrak X$ to be the empty set of variables.

Let us check the third condition of Lemma \ref{nfblemma1}. Let ${\bf U} \in [\![{\bf U}_n]\!]_S$ be a word such that the set $X=\{{_{1{\bf U}_n}x}, {_{1{\bf U}_n}y}_1\}$ is $l_{{\bf U}_n, {\bf U}}$-stable in ${\bf U}_n \approx {\bf U}$, i.e.
${_{1{\bf U}}x} <_{\bf U} {_{1{\bf U}}y}_1$.

 Property (P2) implies that

  ${_{1{\bf U}}x} <_{\bf U} ({_{1{\bf U}}y}_1) <_{\bf U} ({_{2{\bf U}}y}_n) <_{\bf U} \dots <_{\bf U} ({_{2{\bf U}}y}_2) <_{\bf U} ({_{2{\bf U}}y}_1)$.

 Properties (P2) and (P3) imply that

${_{1{\bf U}}x} <_{\bf U} ({_{2{\bf U}}x}) <_{\bf U} ({_{2{\bf U}}y}_n) <_{\bf U} \dots <_{\bf U} ({_{2{\bf U}}y}_2) <_{\bf U} ({_{2{\bf U}}y}_1)$.

Let ${\bf u}$ be a word in less than $n/4$ variables such that $\Theta({\bf u})={\bf U}$ for some substitution $\Theta: \mathfrak A \rightarrow \mathfrak A^+$.
Since the word ${\bf u}$ has less than $n/2$ variables, for some
$c \in \ocs({\bf u})$ and $2 < i < n-1$ both ${_{2{\bf U}}y}_{i+1}$ and ${_{2{\bf U}}y}_{i}$ are contained in $\Theta_{\bf u}(c)$.
Then property (P2) implies that $c$ must be the only occurrence of a linear variable $t$ in ${\bf u}$.

Since $\Theta^{-1}_{\bf u}$ is a homomorphism from $(\ocs({\bf U}), <_{\bf U})$ to $(\ocs({\bf u}), <_{\bf u})$,
we have that

\( \begin{array}{l}
 \Theta^{-1}_{\bf u}({_{1{\bf U}}x})  \le_{\bf u}  \Theta^{-1}_{\bf u}({_{1{\bf U}}y}_1) \le_{\bf u} (_{\bf u}t)  \le_{\bf u} \Theta^{-1}_{\bf u}({_{2{\bf U}}y}_1)
 \end{array} \) and

\( \begin{array}{l}

\Theta^{-1}_{\bf u}({_{1{\bf U}}x}) \le_{\bf u} \Theta^{-1}_{\bf u}({_{2{\bf U}}x}) \le_{\bf u}  (_{\bf u}t) \le_{\bf u} \Theta^{-1}_{\bf u}({_{2{\bf U}}y}_1).
\end{array} \)

Now $\Theta^{-1}_{\bf u}({_{1{\bf U}}x})$ and $\Theta^{-1}_{\bf u}({_{1{\bf U}}y}_1)$ are occurrences of some variables $x$ and $y$ in ${\bf u}$.
By Fact \ref{2occ}, the variables $x$ and $y$ occur at most twice in ${\bf u}$ and if each of them occurs twice in ${\bf u}$ then  $_{1{\bf u}}x = \Theta^{-1}_{\bf u}({_{1{\bf U}}x})$, $_{2{\bf u}}x =\Theta^{-1}_{\bf u}({_{2{\bf U}}x})$, $_{1{\bf u}}y = \Theta^{-1}_{\bf u}({_{1{\bf U}}y}_1)$ and $_{2{\bf u}}y =\Theta^{-1}_{\bf u}({_{2{\bf U}}y}_1)$.

Now let $\bf v$ be an arbitrary word for which $S$ satisfies the identity ${\bf u} \approx {\bf v}$ and ${\bf V} = \Theta({\bf v})$.
Since the word ${\bf u}(x,y,t)$ satisfies all conditions of Lemma \ref{row1}, the set $\{ {_{1{\bf u}}x},  {_{1{\bf u}}y} \}$ is $l_{{\bf u}, {\bf v}}$-stable in
${\bf u} \approx {\bf v}$, i.e. $({_{1{\bf v}}x}) \le_{\bf v}  ({_{1{\bf v}}y})$.
Then Lemma \ref{samelabels} implies that $({_{1{\bf V}}x}) <_{\bf V} ({_{1{\bf V}}y}_1)$.
This means that the set $l_{{\bf U}_n, {\bf U}}(X)=\{{_{1{\bf U}}x}, {_{1{\bf U}}y}_1 \}$ is $l_{{\bf U}, {\bf V}}$-stable in ${\bf U} \approx {\bf V}$. Therefore, the monoid $S$ is non-finitely based  by Lemma \ref{nfblemma1}.

{\bf Row 2.} Here $W = \{yxxty, ytxxy\}$ and \[{\bf U}_n = [Zn]^t \hskip .04in yxx[Zn]y \approx [Zn]^t \hskip .04in xxy[Zn]y = {\bf V}_n.\]
Let $S$ be a monoid such that each word in the set $W$ is an isoterm for $S$ and $S$ satisfies the identity ${\bf U}_n \approx {\bf V}_n$ for each $n>1$.
If ${\bf U} \in [\![{\bf U}_n]\!]_S$ then ${\bf U}$ satisfies the following properties:

(P1) $D_y({\bf U}) = [Zn]^t \hskip .04in xx[Zn]$;

(P2) $occ_{\bf U}(y)=2$ and for each $1 \le i \le n$ we have ${_{2{\bf U}}z_i} <_{\bf U} {_{2{\bf U}}y}$.

Property (P1) follows from the fact that  $D_y({\bf U}_n)= [Zn]^t \hskip .04inxx[Zn] = z_1 t z_2 t \dots z_n t xx z_1 z_2  \dots z_n$ and Lemma \ref{row2}.
Notice that for each $1 \le i \le n$ there is a linear letter in ${\bf U}_n$ such that ${\bf U}_n (z_i, t, y) = z_i ty z_i y$.
Now Property (P2) follows from the fact that the words $\{z_it yy z_i, z_i tz_i, tyy \}$ are isoterms for $S$.

Now take $n>10$ and $X=\{{_{1{\bf U}_n}x}, {_{1{\bf U}_n}y}\}$. Evidently, the set $X=\{{_{1{\bf U}_n}x}, {_{1{\bf U}_n}y}\}$ is $l_{{\bf U}_n, {\bf V}_n}$-unstable in ${\bf U}_n \approx {\bf V}_n$.  We choose $\mathfrak X$ to be the empty set of variables.

Let us check the third condition of Lemma \ref{nfblemma1}. Let ${\bf U} \in [\![{\bf U}_n]\!]_S$ be a word such that the set $X=\{{_{1{\bf U}_n}x}, {_{1{\bf U}_n}y}\}$ is $l_{{\bf U}_n, {\bf U}}$-stable in ${\bf U}_n \approx {\bf U}$, i.e.
${_{1{\bf U}}y} <_{\bf U} {_{1{\bf U}}x}$.

Let ${\bf u}$ be a word in less than $n/4$ variables such that $\Theta({\bf u})={\bf U}$ for some substitution $\Theta: \mathfrak A \rightarrow \mathfrak A^+$.
Since the word ${\bf u}$ has less than $n/2$ variables, for some
$c \in \ocs({\bf u})$ and $2 < i < n-1$ both ${_{2{\bf U}}z}_{i}$ and ${_{2{\bf U}}z}_{i+1}$ are contained in $\Theta_{\bf u}(c)$.
Then Property (P1) implies that $c$ must be the only occurrence of a linear variable $t$ in ${\bf u}$.

Since $\Theta^{-1}_{\bf u}$ is a homomorphism from $(\ocs({\bf U}), <_{\bf U})$ to $(\ocs({\bf u}), <_{\bf u})$ and in view of Property (P2),
we have that

\( \begin{array}{c}
\Theta^{-1}_{\bf u}({_{1{\bf U}}y})  \le_{\bf u}  \Theta^{-1}_{\bf u}({_{1{\bf U}}x}) \le_{\bf u}  \Theta^{-1}_{\bf u}({_{2{\bf U}}x})
\le_{\bf u} (_{\bf u}t)  \le_{\bf u} \Theta^{-1}_{\bf u}({_{2{\bf U}}y}).
\end{array} \)

Now $\Theta^{-1}_{\bf u}({_{1{\bf U}}x})$ and $\Theta^{-1}_{\bf u}({_{1{\bf U}}y})$ are occurrences of some variables $x$ and $y$ in ${\bf u}$.
By Fact \ref{2occ}, variables $x$ and $y$ occur at most twice in ${\bf u}$ and if each of them occurs twice in ${\bf u}$ then  $_{1{\bf u}}x = \Theta^{-1}_{\bf u}({_{1{\bf U}}x})$, $_{2{\bf u}}x =\Theta^{-1}_{\bf u}({_{2{\bf U}}x})$, $_{1{\bf u}}y = \Theta^{-1}_{\bf u}({_{1{\bf U}}y})$ and $_{2{\bf u}}y =\Theta^{-1}_{\bf u}({_{2{\bf U}}y})$.

 Now let $\bf v$ be an arbitrary word for which $S$ satisfies the identity ${\bf u} \approx {\bf v}$ and ${\bf V} = \Theta({\bf v})$.
 If either $x$ or $y$ is linear in ${\bf u}$ then the set $\{{_{1{\bf u}}x},  {_{1{\bf u}}y}\}$ is $l_{{\bf u}, {\bf v}}$-stable in ${\bf u} \approx {\bf v}$ because $yxxty \preceq \{xxt, txx, xtx \}$. If both $x$ and $y$ occur twice in ${\bf u}$ then the set $\{{_{1{\bf u}}x},  {_{1{\bf u}}y}\}$ is $l_{{\bf u}, {\bf v}}$-stable in ${\bf u} \approx {\bf v}$ because ${\bf u}(x,y,t) = yxxty$ is an isoterm for $S$.
 Since the set $\{{_{1{\bf u}}x},  {_{1{\bf u}}y}\}$ is $l_{{\bf u}, {\bf v}}$-stable in ${\bf u} \approx {\bf v}$, we have  $({_{1{\bf v}}y}) \le_{\bf v}  ({_{1{\bf v}}x})$.
Then Lemma \ref{samelabels} implies that $({_{1{\bf V}}y}) <_{\bf V} ({_{1{\bf V}}x})$.
This means that the set $l_{{\bf U}_n, {\bf U}}(X)=\{{_{1{\bf U}}x}, {_{1{\bf U}}y} \}$ is $l_{{\bf U}, {\bf V}}$-stable in ${\bf U} \approx {\bf V}$. Therefore, the monoid $S$ is non-finitely based  by Lemma \ref{nfblemma1}.

{\bf Row 3.}
 Here $W = \{xtxyty\}$ and \[{\bf U}_n = [ZPn]^t \hskip .04in x[ZQn] xy [PRn]y \hskip .04in ^t[QRn] \approx [ZPn]^t \hskip .04in x[ZQn] yx [PRn]y \hskip .04in ^t[QRn] = {\bf V}_n.\]

Let $S$ be a monoid such that the word $xtxyty$ is an isoterm for $S$ and $S$ satisfies the identity ${\bf U}_n \approx {\bf V}_n$ for each $n>1$.
If ${\bf U} \in [\![{\bf U}_n]\!]_S$ then Lemma \ref{row3} implies that ${\bf U}$ satisfies the following property:

(P) $D_{\{x,y\}}({\bf U}) = [ZPn]^t \hskip .04in [ZQn][PRn] \hskip .04in ^t[QRn]$.

Now take $n>10$, $X=\{{_{1{\bf U}_n}x}, {_{2{\bf U}_n}z}_1, {_{1{\bf U}_n}q}_n, {_{2{\bf U}_n}x},
{_{1{\bf U}_n}y}, {_{2{\bf U}_n}p}_1, {_{1{\bf U}_n}r}_n, {_{2{\bf U}_n}y} \}$ and $\mathfrak X = \{x, y\}$.  Evidently, the set $X$ is $l_{{\bf U}_n, {\bf V}_n}$-unstable in ${\bf U}_n \approx {\bf V}_n$.

Let us check the third condition of Lemma \ref{nfblemma1}. Let ${\bf U} \in [\![{\bf U}_n]\!]_S$ be a word with $occ_{\bf U}(x)=occ_{\bf U}(y)=2$ such that
the set $X$ is $l_{{\bf U}_n, {\bf U}}$-stable in ${\bf U}_n \approx {\bf U}$, i.e.
$({_{1{\bf U}}x}) <_{\bf U} ({_{2{\bf U}}z}_1) <_{\bf U} ({_{1{\bf U}}q}_n) <_{\bf U} ({_{2{\bf U}}x}) <_{\bf U}
({_{1{\bf U}}y}) <_{\bf U} ({_{2{\bf U}}p}_1) <_{\bf U} ({_{1{\bf U}}r_n}) <_{\bf U} ({_{2{\bf U}}y})$.

Let ${\bf u}$ be a word in less than $n/4$ variables such that $\Theta({\bf u})={\bf U}$ for some substitution $\Theta: \mathfrak A \rightarrow \mathfrak A^+$.
Since the word ${\bf u}$ has less than $n/2$ variables, for some
$c \in \ocs({\bf u})$ and $2 < i < n-1$ both ${_{2{\bf U}}z}_{i}$ and ${_{2{\bf U}}z}_{i+1}$ are contained in $\Theta_{\bf u}(c)$.
Then Property (P) implies that $c$ must be the only occurrence of a linear variable $t_1$ in ${\bf u}$.
Similarly, the word ${\bf u}$ contains a linear letter $t_2$ such that $\Theta_{\bf u}( {_{\bf u}t}_2)$ contains both ${_{2{\bf U}}r}_{i}$ and ${_{2{\bf U}}r}_{i+1}$.

Since $\Theta^{-1}_{\bf u}$ is a homomorphism from $(\ocs({\bf U}), <_{\bf U})$ to $(\ocs({\bf u}), <_{\bf u})$,
we have that $\Theta^{-1}_{\bf u}({_{1{\bf U}}x}) \le_{\bf u} \Theta^{-1}_{\bf u}({_{2{\bf U}}z}_1) \le_{\bf u} (_{\bf u}t_1) \le_{\bf u} \Theta^{-1}_{\bf u}({_{1{\bf U}}q}_n)  \le_{\bf u}  \Theta^{-1}_{\bf u}({_{2{\bf U}}x}) \le_{\bf u}  \Theta^{-1}_{\bf u}({_{1{\bf U}}y}) \le_{\bf u} \Theta^{-1}_{\bf u}({_{2{\bf U}}p}_1) \le_{\bf u} (_{\bf u}t_2) \le_{\bf u} \Theta^{-1}_{\bf u}({_{1{\bf U}}r}_n)  \le_{\bf u} \Theta^{-1}_{\bf u}({_{2{\bf U}}y})$.

\begin{claim} \label{r3} The word ${\bf u} (\Theta^{-1}(x) \cup \Theta^{-1}(y) \cup \{t_1, t_2\})$
is an isoterm for $S$.
\end{claim}

\begin{proof} In view of Fact \ref{2occ}, there are only two possibilities for the set $\Theta^{-1}(x)$ and two possibilities for the set $\Theta^{-1}(y)$.
If the set  $\Theta^{-1}(x)$ contains a variable $x$ with $occ_{\bf u}(x)=2$ and the set $\Theta^{-1}(y)$ contains a variable $y$ with $occ_{\bf u}(y)=2$ then
${\bf u}(x,y,t_1,t_2) = xt_1xyt_2y$ is an isoterm for $S$. If either the set  $\Theta^{-1}(x)$ or the set $\Theta^{-1}(y)$ contains a variable that is linear in ${\bf u}$ then we have $xtxyty \preceq \{xxt, txx, xtx \} \preceq {\bf u} (\Theta^{-1}(x) \cup \Theta^{-1}(y) \cup \{t_1, t_2\})$.
\end{proof}

Now let $\bf v$ be an arbitrary word for which $S$ satisfies the identity ${\bf u} \approx {\bf v}$ and ${\bf V} = \Theta({\bf v})$.
Claim \ref{r3} implies that the set of variables $\Theta^{-1}(\{x, y\}) = \Theta^{-1}(x) \cup \Theta^{-1}(y)$
is stable in  ${\bf u} \approx {\bf v}$. Consequently, the set $\{x, y\}$ is stable in ${\bf U} \approx {\bf V}$ by Lemma \ref{stvar}.
 In particular, this means that $occ_{\bf V}(x)=occ_{\bf V}(y)=2$, i.e. each variable in $\mathfrak X = \{x, y\}$
is stable in ${\bf U} \approx {\bf V}$.

\begin{claim} \label{r31} The set

 $\{\Theta^{-1}_{\bf u}({_{1{\bf U}}x}), \Theta^{-1}_{\bf u}({_{2{\bf U}}z}_1), \Theta^{-1}_{\bf u}({_{1{\bf U}}q}_n), \Theta^{-1}_{\bf u}({_{2{\bf U}}x}), \Theta^{-1}_{\bf u}({_{1{\bf U}}y}), \Theta^{-1}_{\bf u}({_{2{\bf U}}p}_1), \Theta^{-1}_{\bf u}({_{1{\bf U}}r}_n), \Theta^{-1}_{\bf u}({_{2{\bf U}}y}) \}$ is $l_{{\bf u}, {\bf v}}$-stable in  ${\bf u} \approx {\bf v}$.

\end{claim}

\begin{proof} In view of Claim \ref{r3} and Lemma \ref{row3}, it is only left to prove that the following sets are $l_{{\bf u}, {\bf v}}$-stable in ${\bf u} \approx {\bf v}$:
$\{\Theta^{-1}_{\bf u}({_{1{\bf U}}x}), \Theta^{-1}_{\bf u}({_{2{\bf U}}z}_1)\}$, $\{\Theta^{-1}_{\bf u}({_{1{\bf U}}q}_n), \Theta^{-1}_{\bf u}({_{2{\bf U}}x})\}$,
$\{\Theta^{-1}_{\bf u}({_{1{\bf U}}y}) , \Theta^{-1}_{\bf u}({_{2{\bf U}}p}_1)\}$, $\{\Theta^{-1}_{\bf u}({_{1{\bf U}}r}_n) , \Theta^{-1}_{\bf u}({_{2{\bf U}}y})\}$.

We only show that the set $\{ \Theta^{-1}_{\bf u}({_{1{\bf U}}x}),
\Theta^{-1}_{\bf u}({_{2{\bf U}}z}_1)\}$ is $l_{{\bf u}, {\bf v}}$-stable in ${\bf u} \approx {\bf v}$. (Proofs for the other three sets are similar.)
Indeed, $\Theta^{-1}_{\bf u}({_{1{\bf U}}x})$ and $\Theta^{-1}_{\bf u}({_{2{\bf U}}z}_1)$ are occurrences of some variables $x$ and $z$ in ${\bf u}$.
By Fact \ref{2occ}, the variables $x$ and $z$ occur at most twice in ${\bf u}$ and if each of them occurs twice in ${\bf u}$ then  $_{1{\bf u}}x = \Theta^{-1}_{\bf u}({_{1{\bf U}}x})$, $_{2{\bf u}}x =\Theta^{-1}_{\bf u}({_{2{\bf U}}x})$, $_{1{\bf u}}z = \Theta^{-1}_{\bf u}({_{1{\bf U}}z}_1)$ and $_{2{\bf u}}z =\Theta^{-1}_{\bf u}({_{2{\bf U}}z}_1)$. If either $x$ or $z$ is linear in ${\bf u}$ then the set $\{{_{\bf u}x},  {_{2{\bf u}}z}\}$ or $\{{_{1{\bf u}}x},  {_{\bf u}z}\}$ is $l_{{\bf u}, {\bf v}}$-stable in ${\bf u} \approx {\bf v}$ because $xtxyty \preceq \{xxt, txx, xtx \}$. If both $x$ and $z$ occur twice in ${\bf u}$ then the set $\{{_{1{\bf u}}x},  {_{2{\bf u}}z}\}$ is $l_{{\bf u}, {\bf v}}$-stable in ${\bf u} \approx {\bf v}$ because ${\bf u}(x,z,t) = ztxzt_1x$ is an isoterm for $S$.
\end{proof}

Claim \ref{r31} and  Lemma \ref{samelabels} imply that the set

 $l_{{\bf U}_n, {\bf U}}(X)=\{ {_{1{\bf U}}x}, {_{2{\bf U}}z}_1, {_{1{\bf U}}q}_n, {_{2{\bf U}}x},
{_{1{\bf U}}y}, {_{2{\bf U}}p}_1, {_{1{\bf U}}r}_n, {_{2{\bf U}}y} \}$ is $l_{{\bf U}, {\bf V}}$-stable in ${\bf U} \approx {\bf V}$.
 Therefore, the  monoid $S$ is non-finitely based  by Lemma \ref{nfblemma1}.

{\bf Row 4.}  Here $W = \{xxyy, xytytx\}$ and \[ {\bf U}_n = xytyz_1^2z_2^2 \dots z_n^2x \approx yxtyz_1^2z_2^2 \dots z_n^2x = {\bf V}_n. \]
Let $S$ be a monoid such that each word in $W$ is an isoterm for $S$ and $S$ satisfies the identity ${\bf U}_n \approx {\bf V}_n$ for each $n>1$.

\begin{claim} \label{row4}
If ${\bf U} \in [\![{\bf U}_n]\!]_S$ then ${\bf U}$ satisfies the following properties:

(P1) $occ_{\bf U}(x)=2$;

(P2) $D_x({\bf U}) = yty z_1^2 \dots z_n^2$;

(P3)  ${_{1{\bf U}}x} <_{\bf U} {_{\bf U}t}$;

(P4) ${_{1{\bf U}}z_n} <_{\bf U} {_{2{\bf U}}x}$.
\end{claim}

\begin{proof} Property (P1) follows from the fact that $xx$ is an isoterm for $S$. Property (P2) follows from the fact that  $D_x({\bf U}_n)= yty z_1^2 \dots z_n^2  \succeq xxyy$. Property (P3) follows from the fact that $xtx$ is an isoterm for $S$. If ${_{2{\bf U}}x} <_{\bf U} {_{1{\bf U}}z_n}$ then the monoid $S$ would satisfy
the identity ${\bf U}(x,t,z_n)=xtxz_nz_n \approx xtz_nz_nx = {\bf U}_n(x,t, z_n)$. Since this identity is false in $S$, Property (P4) is verified.
\end{proof}

Now take $n>10$ and $X= \{( {_{1{\bf U}_n}x}, {_{1{\bf U}_n}y})\}$. Evidently, the set $X= \{( {_{1{\bf U}_n}x}, {_{1{\bf U}_n}y})\}$ is $l_{{\bf U}_n, {\bf V}_n}$-unstable in ${\bf U}_n \approx {\bf V}_n$. We choose $\mathfrak X$ to be the empty set of variables.

Let us check the third condition of Lemma \ref{nfblemma1}. Let ${\bf U} \in [\![{\bf U}_n]\!]_S$ be a word such that the set $X=\{{_{1{\bf U}_n}x}, {_{1{\bf U}_n}y}\}$ is $l_{{\bf U}_n, {\bf U}}$-stable in ${\bf U}_n \approx {\bf U}$, i.e.
${_{1{\bf U}}x} <_{\bf U} {_{1{\bf U}}y}$.

Let ${\bf u}$ be a word in less than $n/4$ variables such that $\Theta({\bf u})={\bf U}$ for some substitution $\Theta: \mathfrak A \rightarrow \mathfrak A^+$.
Since the word ${\bf u}$ has less than $n/2$ variables, for some
$c \in \ocs({\bf u})$ and $2 < i < n-1$ both ${_{1{\bf U}}z}_{i}$ and ${_{1{\bf U}}z}_{i+1}$ are contained in $\Theta_{\bf u}(c)$.
Then Property (P2) implies that $c$ must be the only occurrence of a linear variable $t_1$ in ${\bf u}$.

Since $\Theta^{-1}_{\bf u}$ is a homomorphism from $(\ocs({\bf U}), <_{\bf U})$ to $(\ocs({\bf u}), <_{\bf u})$, Properties (P1)--(P4) of Claim \ref{row4}
imply that $\Theta^{-1}_{\bf u}({_{1{\bf U}}x})  \le_{\bf u}  \Theta^{-1}_{\bf u}({_{1{\bf U}}y}) \le_{\bf u} \Theta^{-1}_{\bf u}({_{2{\bf U}}y})
\le_{\bf u} (_{\bf u}t)  \le_{\bf u} \Theta^{-1}_{\bf u}({_{2{\bf U}}x})$.

Now $\Theta^{-1}_{\bf u}({_{1{\bf U}}x})$ and $\Theta^{-1}_{\bf u}({_{1{\bf U}}y})$ are occurrences of some variables $x$ and $y$ in ${\bf u}$.
By Fact \ref{2occ}, variables $x$ and $y$ occur at most twice in ${\bf u}$ and if each of them occurs twice in ${\bf u}$ then  $_{1{\bf u}}x = \Theta^{-1}_{\bf u}({_{1{\bf U}}x})$, $_{2{\bf u}}x =\Theta^{-1}_{\bf u}({_{2{\bf U}}x})$, $_{1{\bf u}}y = \Theta^{-1}_{\bf u}({_{1{\bf U}}y})$ and $_{2{\bf u}}y =\Theta^{-1}_{\bf u}({_{2{\bf U}}y})$.

 Now let $\bf v$ be an arbitrary word for which $S$ satisfies the identity ${\bf u} \approx {\bf v}$ and ${\bf V} = \Theta({\bf v})$.

 If either $x$ or $y$ is linear in ${\bf u}$ then set $\{{_{1{\bf u}}x},  {_{1{\bf u}}y}\}$ is $l_{{\bf u}, {\bf v}}$-stable in ${\bf u} \approx {\bf v}$ because $xxyy \preceq \{xxt, txx, xtx \}$. If both $x$ and $y$ occur twice in ${\bf u}$ then the set $\{{_{1{\bf u}}x},  {_{1{\bf u}}y}\}$ is $l_{{\bf u}, {\bf v}}$-stable in ${\bf u} \approx {\bf v}$ because
 ${\bf u}(x,y,t,t_1) = xytyt_1x$ is an isoterm for $S$.
 Since the set $\{{_{1{\bf u}}x},  {_{1{\bf u}}y}\}$ is $l_{{\bf u}, {\bf v}}$-stable in ${\bf u} \approx {\bf v}$, we have  $({_{1{\bf v}}x}) \le_{\bf v}  ({_{1{\bf v}}y})$.
Then Lemma \ref{samelabels} implies that $({_{1{\bf V}}y}) <_{\bf V} ({_{1{\bf V}}x})$.
This means that the set $l_{{\bf U}_n, {\bf U}}(X)=\{{_{1{\bf U}}x}, {_{1{\bf U}}y} \}$ is $l_{{\bf U}, {\bf V}}$-stable in ${\bf U} \approx {\bf V}$. Therefore, the monoid $S$ is non-finitely based  by Lemma \ref{nfblemma1}.

{\bf Row 5.}  Here $W = \{ xtyxty, xytxy, xytyx\}$ and \[{\bf U}_n = xy[Zn]yxt[nZ]  \approx yx[Zn]xyt[nZ] = {\bf V}_n.\]
Let $S$ be a monoid such that each word in the set $W$ is an isoterm for $S$ and $S$ satisfies the identity ${\bf U}_n \approx {\bf V}_n$ for each $n>1$.
If ${\bf U} \in [\![{\bf U}_n]\!]_S$ then ${\bf U}$ satisfies the following property:

(P) $D_{x,y}({\bf U}) = [Zn]t[nZ] = z_1 z_2 \dots z_n t z_n \dots z_2 z_1$.

(Property (P) follows from the fact $D_{x,y}({\bf U}_n)= [Zn]t[nZ] \succeq xytyx$.)

Now take $n>10$, $X=\{{_{1{\bf U}_n}x}, {_{1{\bf U}_n}y}, {_{1{\bf U}_n}z_1}, {_{1{\bf U}_n}z_n, {_{2{\bf U}_n}y}, {_{2{\bf U}_n}x}}\}$ and $\mathfrak X = \{x,y\}$.
 Evidently, the set $X$ is $l_{{\bf U}_n, {\bf V}_n}$-unstable in ${\bf U}_n \approx {\bf V}_n$.

Let us check the third condition of Lemma \ref{nfblemma1}. Let ${\bf U} \in [\![{\bf U}_n]\!]_S$ be a word with $occ_{\bf U}(x)=occ_{\bf U}(y)=2$ such that
the set $X$ is $l_{{\bf U}_n, {\bf U}}$-stable in ${\bf U}_n \approx {\bf U}$, i.e.
$({_{1{\bf U}}x}) <_{\bf U} ({_{1{\bf U}}y}) <_{\bf U} ({_{1{\bf U}}z}_1) <_{\bf U} ({_{1{\bf U}}z_n}) <_{\bf U}
({_{2{\bf U}}y}) <_{\bf U} ({_{2{\bf U}}x})$.

Let ${\bf u}$ be a word in less than $n/4$ variables such that $\Theta({\bf u})={\bf U}$ for some substitution $\Theta: \mathfrak A \rightarrow \mathfrak A^+$.
Since the word ${\bf u}$ has less than $n/2$ variables, for some
$c \in \ocs({\bf u})$ and $2 < i < n-1$ both ${_{1{\bf U}}z}_{i}$ and ${_{1{\bf U}}z}_{i+1}$ are contained in $\Theta_{\bf u}(c)$.
Then Property (P2) implies that $c$ must be the only occurrence of a linear variable $t$ in ${\bf u}$.

Since $\Theta^{-1}_{\bf u}$ is a homomorphism from $(\ocs({\bf U}), <_{\bf U})$ to $(\ocs({\bf u}), <_{\bf u})$,
we have that $\Theta^{-1}_{\bf u}({_{1{\bf U}}x})  \le_{\bf u}  \Theta^{-1}_{\bf u}({_{1{\bf U}}y}) \le_{\bf u}  \Theta^{-1}_{\bf u}({_{1{\bf U}}z}_1)
\le_{\bf u} (_{\bf u}t)  \le_{\bf u} \Theta^{-1}_{\bf u}({_{1{\bf U}}z}_n) \le_{\bf u}  \Theta^{-1}_{\bf u}({_{2{\bf U}}y})  \le_{\bf u}  \Theta^{-1}_{\bf u}({_{2{\bf U}}x})$.

\begin{claim} \label{r5} The word ${\bf u} (\Theta^{-1}(x) \cup \Theta^{-1}(y) \cup \{t \})$
is an isoterm for $S$.
\end{claim}

\begin{proof} In view of Fact \ref{2occ}, there are only two possibilities for the set $\Theta^{-1}(x)$ and two possibilities for the set $\Theta^{-1}(y)$.
If the set  $\Theta^{-1}(x)$ contains a variable $x$ with $occ_{\bf u}(x)=2$ and the set $\Theta^{-1}(y)$ contains a variable $y$ with $occ_{\bf u}(x)=2$ then
${\bf u}(x,y,t) = xytyx$ is an isoterm for $S$. If either the set  $\Theta^{-1}(x)$ or the set $\Theta^{-1}(y)$ contains a variable that is linear in ${\bf u}$
then we have $xytyx \preceq {\bf u} (\Theta^{-1}(x) \cup \Theta^{-1}(y) \cup \{t\})$.
\end{proof}

 Now let $\bf v$ be an arbitrary word for which $S$ satisfies the identity ${\bf u} \approx {\bf v}$ and ${\bf V} = \Theta({\bf v})$.
Claim \ref{r5} implies that the set of variables $\Theta^{-1}(\{x, y\}) = \Theta^{-1}(x) \cup \Theta^{-1}(y)$
is stable in  ${\bf u} \approx {\bf v}$. Consequently, the set $\{x, y\}$ is stable in ${\bf U} \approx {\bf V}$ by Lemma \ref{stvar}.
 In particular, this means that $occ_{\bf V}(x)=occ_{\bf V}(y)=2$, i.e. each variable in $\mathfrak X = \{x, y\}$
is stable in ${\bf U} \approx {\bf V}$.

\begin{claim} \label{r51} The set $=\{\Theta^{-1}_{\bf u}({_{1{\bf U}}x}), \Theta^{-1}_{\bf u}({_{1{\bf U}}y}), \Theta^{-1}_{\bf u}({_{1{\bf U}}z}_1), \Theta^{-1}_{\bf u}({_{1{\bf U}}z}_n), \Theta^{-1}_{\bf u}({_{2{\bf U}}y}), \Theta^{-1}_{\bf u}({_{2{\bf U}}x})\}$

is $l_{{\bf u}, {\bf v}}$-stable in  ${\bf u} \approx {\bf v}$.

\end{claim}

\begin{proof} In view of Claim \ref{r5} and the fact that $xytyx \preceq [Zn]t[nZ]$, it is only left to prove that the following sets are $l_{{\bf u}, {\bf v}}$-stable in ${\bf u} \approx {\bf v}$:
$\{\Theta^{-1}_{\bf u}({_{1{\bf U}}y}), \Theta^{-1}_{\bf u}({_{1{\bf U}}z}_1)\}$, $\{\Theta^{-1}_{\bf u}({_{1{\bf U}}z}_n), \Theta^{-1}_{\bf u}({_{2{\bf U}}y})\}$.

Indeed, $\Theta^{-1}_{\bf u}({_{1{\bf U}}y})$ and $\Theta^{-1}_{\bf u}({_{1{\bf U}}z}_1)$ are occurrences of some variables $y$ and $z$ in ${\bf u}$.
By Fact \ref{2occ}, the variables $y$ and $z$ occur at most twice in ${\bf u}$ and if each of them occurs twice in ${\bf u}$ then  $_{1{\bf u}}y = \Theta^{-1}_{\bf u}({_{1{\bf U}}y})$, $_{2{\bf u}}y =\Theta^{-1}_{\bf u}({_{2{\bf U}}y})$, $_{1{\bf u}}z = \Theta^{-1}_{\bf u}({_{1{\bf U}}z}_1)$ and $_{2{\bf u}}z =\Theta^{-1}_{\bf u}({_{2{\bf U}}z}_1)$. If either $y$ or $z$ is linear in ${\bf u}$ then the set $\{{_{1{\bf u}}y},  {_{1{\bf u}}z}\}$ is $l_{{\bf u}, {\bf v}}$-stable in ${\bf u} \approx {\bf v}$ because
$xytyx \preceq xtx$. If both $y$ and $z$ occur twice in ${\bf u}$ then the set $\{{_{1{\bf u}}y},  {_{1{\bf u}}z}\}$ is $l_{{\bf u}, {\bf v}}$-stable in ${\bf u} \approx {\bf v}$ because ${\bf u}(y,z,t) = yztyz$.

Similarly, $\Theta^{-1}_{\bf u}({_{2{\bf U}}y})$ and $\Theta^{-1}_{\bf u}({_{1{\bf U}}z}_n)$ are occurrences of some variables $y$ and $z$ in ${\bf u}$.
By Fact \ref{2occ}, the variables $y$ and $z$ occur at most twice in ${\bf u}$ and if each of them occurs twice in ${\bf u}$ then  $_{1{\bf u}}y = \Theta^{-1}_{\bf u}({_{1{\bf U}}y})$, $_{2{\bf u}}y =\Theta^{-1}_{\bf u}({_{2{\bf U}}y})$, $_{1{\bf u}}z = \Theta^{-1}_{\bf u}({_{1{\bf U}}z}_n)$ and $_{2{\bf u}}z =\Theta^{-1}_{\bf u}({_{2{\bf U}}z}_n)$. If either $y$ or $z$ is linear in ${\bf u}$ then the set $\{{_{2{\bf u}}y},  {_{1{\bf u}}z}\}$ is $l_{{\bf u}, {\bf v}}$-stable in ${\bf u} \approx {\bf v}$ because
$xytyx \preceq xtx$. If both $y$ and $z$ occur twice in ${\bf u}$ then the set $\{{_{2{\bf u}}y},  {_{1{\bf u}}z}\}$ is $l_{{\bf u}, {\bf v}}$-stable in ${\bf u} \approx {\bf v}$ because ${\bf u}(y,z,t, t') = ytzyt'z$. (Here $_{\bf u}t' = \Theta^{-1}_{\bf u}({_{\bf U}t})$.)
\end{proof}

Claim \ref{r51} and  Lemma \ref{samelabels} imply that the set $l_{{\bf U}_n, {\bf U}}(X)= \{{_{1{\bf U}}x}, {_{1{\bf U}}y}, {_{1{\bf U}}z_1}, {_{1{\bf U}}z_n, {_{2{\bf U}}y}, {_{2{\bf U}}x}}\}$
is $l_{{\bf U}, {\bf V}}$-stable in ${\bf U} \approx {\bf V}$. Therefore, the monoid $S$ is non-finitely based  by Lemma \ref{nfblemma1}.

{\bf Row 6.} Similar to  Row 5.

{\bf Row 7.}  Here $W = \{xtxyty, xyyx \}$ and \[{\bf U}_n = [Xn][nX][Yn][nY] \approx [Yn][nY][Xn][nX] = {\bf V}_n.\]
Let $S$ be a monoid such that each word in the set $W$ is an isoterm for $S$ and $S$ satisfies the identity ${\bf U}_n \approx {\bf V}_n$ for each $n>1$.
If ${\bf U} \in [\![{\bf U}_n]\!]_S$ then ${\bf U}$ satisfies the following properties:

(P1) $D_{\{x_1, \dots, x_n\}}({\bf U}) = [Yn][nY]= y_1y_2 \dots y_n y_n \dots y_2 y_1$;

(P2) $D_{\{y_1, \dots, y_n\}}({\bf U}) = [Xn][nX]= x_1x_2 \dots x_n x_n \dots x_2 x_1$.

(Properties (P1) and (P2) follow from the fact that  $D_{\{y_1, \dots, y_n\}}({\bf U}_n) = [Xn][nX] = x_1x_2 \dots x_n x_n \dots x_2 x_1 \succeq xyyx$.)

Now take $n>10$ and $X=\{{_{2{\bf U}_n}x}_1, {_{1{\bf U}_n}y}_1\}$.  Evidently, the set $X=\{{_{2{\bf U}_n}x}_1, {_{1{\bf U}_n}y}_1\}$ is $l_{{\bf U}_n, {\bf V}_n}$-unstable in ${\bf U}_n \approx {\bf V}_n$. We choose $\mathfrak X$ to be the empty set of variables.

Let us check the third condition of Lemma \ref{nfblemma1}. Let ${\bf U} \in [\![{\bf U}_n]\!]_S$ be a word such that the set $X=\{{_{2{\bf U}_n}x}_1, {_{1{\bf U}_n}y}_1\}$ is $l_{{\bf U}_n, {\bf U}}$-stable in ${\bf U}_n \approx {\bf U}$, i.e.
${_{2{\bf U}}x}_1 <_{\bf U} {_{1{\bf U}}y}_1$.

Let ${\bf u}$ be a word in less than $n/4$ variables such that $\Theta({\bf u})={\bf U}$ for some substitution $\Theta: \mathfrak A \rightarrow \mathfrak A^+$.
Since the word ${\bf u}$ has less than $n/2$ variables, for some
$c \in \ocs({\bf u})$ and $2 < i < n-1$ both ${_{2{\bf U}}x}_{i}$ and ${_{2{\bf U}}x}_{i+1}$ are contained in $\Theta_{\bf u}(c)$.
Then Property (P2) implies that $c$ must be the only occurrence of a linear variable $t_1$ in ${\bf u}$.
Similarly, the word ${\bf u}$ contains a linear letter $t_2$ such that $\Theta_{\bf u}( {_{\bf u}t}_2)$ contains both ${_{2{\bf U}}y}_{i}$ and ${_{2{\bf U}}y}_{i+1}$.

Since $\Theta^{-1}_{\bf u}$ is a homomorphism from $(\ocs({\bf U}), <_{\bf U})$ to $(\ocs({\bf u}), <_{\bf u})$,
we have that $\Theta^{-1}_{\bf u}({_{1{\bf U}}x}_1) \le_{\bf u} (_{\bf u}t_1) \le_{\bf u}  \Theta^{-1}_{\bf u}({_{2{\bf U}}x}_1) \le_{\bf u}  \Theta^{-1}_{\bf u}({_{1{\bf U}}y}_1) \le_{\bf u} (_{\bf u}t_2)  \le_{\bf u} \Theta^{-1}_{\bf u}({_{2{\bf U}}y}_1)$.

Now $\Theta^{-1}_{\bf u}({_{2{\bf U}}x}_1)$ and $\Theta^{-1}_{\bf u}({_{1{\bf U}}y}_1)$ are occurrences of some variables $x$ and $y$ in ${\bf u}$.
By Fact \ref{2occ}, the variables $x$ and $y$ occur at most twice in ${\bf u}$ and if each of them occurs twice in ${\bf u}$ then  $_{1{\bf u}}x = \Theta^{-1}_{\bf u}({_{1{\bf U}}x}_1)$, $_{2{\bf u}}x =\Theta^{-1}_{\bf u}({_{2{\bf U}}x}_1)$, $_{1{\bf u}}y = \Theta^{-1}_{\bf u}({_{1{\bf U}}y}_1)$ and $_{2{\bf u}}y =\Theta^{-1}_{\bf u}({_{2{\bf U}}y}_1)$.

 Now let $\bf v$ be an arbitrary word for which $S$ satisfies the identity ${\bf u} \approx {\bf v}$ and ${\bf V} = \Theta({\bf v})$.
 If either $x$ or $y$ is linear in ${\bf u}$ then the set $\{{_{1{\bf u}}x},  {_{1{\bf u}}y}\}$ is $l_{{\bf u}, {\bf v}}$-stable in ${\bf u} \approx {\bf v}$ because $xtxyty \preceq xtx$. If both $x$ and $y$ occur twice in ${\bf u}$ then the set $\{{_{2{\bf u}}x},  {_{1{\bf u}}y}\}$ is $l_{{\bf u}, {\bf v}}$-stable in ${\bf u} \approx {\bf v}$ because
 ${\bf u}(x,y,t) = xtxyty$ is an isoterm for $S$.
 Since the set $\{{_{2{\bf u}}x},  {_{1{\bf u}}y}\}$ is $l_{{\bf u}, {\bf v}}$-stable in ${\bf u} \approx {\bf v}$, we have  $({_{2{\bf v}}x}) \le_{\bf v}  ({_{1{\bf v}}y})$.
Then Lemma \ref{samelabels} implies that $({_{2{\bf V}}x}_1) <_{\bf V} ({_{1{\bf V}}y}_1)$.
This means that the set $l_{{\bf U}_n, {\bf U}}(X)=\{{_{2{\bf U}}x}_1, {_{1{\bf U}}y}_1 \}$ is $l_{{\bf U}, {\bf V}}$-stable in ${\bf U} \approx {\bf V}$. Therefore, the monoid $S$ is non-finitely based  by Lemma \ref{nfblemma1}.

{\bf Row 8.} Here $W = \{ xxyy \} \cup \{ytyx^dtx^{m-d}, x^{m-d}tx^dyty | 0 <d <m\}$ for some $m>2$ and ${\bf U}_n = yt_1x^{m-1}yp_1^2 \dots p_n^2zxt_2z \approx = yt_1x^{m}yp_1^2 \dots p_n^2zt_2z = {\bf V}_n$.
Let $S$ be a monoid such that each word in the set $W$ is an isoterm for $S$ and $S$ satisfies the identity ${\bf U}_n \approx {\bf V}_n$ for each $n>1$.
If ${\bf U} \in [\![{\bf U}_n]\!]_S$ then ${\bf U}$ satisfies the following property:

(P) $D_x({\bf U}) = yt_1 yp_1^2 \dots p_n^2zt_2z$.

(Property (P) follows from the fact that $D_x({\bf U}_n)= yt_1 yp_1^2 \dots p_n^2zt_2z \succeq xxyy$.)

Now take $n>10$, $X=\{{{_{{\bf U}_n}t_1}}, {_{1{\bf U}_n}x}, {_{2{\bf U}_n}y}, {_{1{\bf U}_n}z, {_{m{\bf U}_n}x}, {_{{\bf U}_n}t}_2}\}$ and $\mathfrak X = \{x \}$.
Evidently, the set $X$ is $l_{{\bf U}_n, {\bf V}_n}$-unstable in ${\bf U}_n \approx {\bf V}_n$.

Let us check the third condition of Lemma \ref{nfblemma1}. Let ${\bf U} \in [\![{\bf U}_n]\!]_S$ be a word with $occ_{\bf U}(x)=m$ such that
the set $X$ is $l_{{\bf U}_n, {\bf V}_n}$-stable in ${\bf U}_n \approx {\bf U}$, i.e.
$({_{{\bf U}}t_1}) <_{\bf U} ({_{1{\bf U}}x}) <_{\bf U} ({_{2{\bf U}}y}) <_{\bf U} ({_{1{\bf U}}z}) <_{\bf U}
({_{m{\bf U}}x}) <_{\bf U} ({_{{\bf U}}t_2})$.

Let ${\bf u}$ be a word in less than $n/4$ variables such that $\Theta({\bf u})={\bf U}$ for some substitution $\Theta: \mathfrak A \rightarrow \mathfrak A^+$.
Since the word ${\bf u}$ has less than $n/2$ variables, for some
$c \in \ocs({\bf u})$ and $2 < i < n-1$ both ${_{1{\bf U}}p}_{i}$ and ${_{1{\bf U}}p}_{i+1}$ are contained in $\Theta_{\bf u}(c)$.
Then Property (P2) implies that $c$ must be the only occurrence of a linear variable $t$ in ${\bf u}$.

Since $\Theta^{-1}_{\bf u}$ is a homomorphism from $(\ocs({\bf U}), <_{\bf U})$ to $(\ocs({\bf u}), <_{\bf u})$,
we have that $\Theta^{-1}_{\bf u}({_{{\bf U}}t}_1)  \le_{\bf u}  \Theta^{-1}_{\bf u}({_{1{\bf U}}x}) \le_{\bf u}  \Theta^{-1}_{\bf u}({_{2{\bf U}}y})
\le_{\bf u} (_{\bf u}t)  \le_{\bf u} \Theta^{-1}_{\bf u}({_{1{\bf U}}z}) \le_{\bf u}  \Theta^{-1}_{\bf u}({_{m{\bf U}}x})  \le_{\bf u}  \Theta^{-1}_{\bf u}({_{{\bf U}}t_2})$.

\begin{claim} \label{r8} Each variable in $\Theta^{-1}(x)$ is stable in ${\bf u}$ with respect to $S$.
\end{claim}

\begin{proof} If a variable $x \in \Theta^{-1}(x)$ occurs in ${\bf u}$ exactly $m$ times, then
 ${_{1{\bf u}}x} =\Theta^{-1}_{\bf u}({_{1{\bf U}}x})$ and ${_{m{\bf u}}x} =\Theta^{-1}_{\bf u}({_{m{\bf U}}x})$. Since $({_{1{\bf u}}x}) <_{\bf u} (_{\bf u}t) < _{\bf u} ({_{m{\bf u}}x})$ we have that ${\bf u}(x,t) = x^dtx^{m-d}$ for some  $0 < d <m$. Since this word is an isoterm for $S$, the variable $x$ is stable in ${\bf u}$ with respect to $S$. If a variable $y \in \Theta^{-1}(x)$ occurs in ${\bf u}$ less than $m$ times, then $y$ is stable in ${\bf u}$ with respect to $S$ because the word $y^{m-1}$ is an isoterm for $S$. \end{proof}

Now let $\bf v$ be an arbitrary word for which $S$ satisfies the identity ${\bf u} \approx {\bf v}$ and ${\bf V} = \Theta({\bf v})$.
Then by Claim \ref{r8} each variable in $\Theta^{-1}(x)$ is stable in ${\bf u} \approx {\bf v}$.
Therefore, the variable $x$ is stable in ${\bf U} \approx {\bf V}$, i.e. $occ_{\bf U}(x)=occ_{\bf V}(x)=m$.

\begin{claim} \label{r81} The set $l_{{\bf U}_n, {\bf U}}(X)= \{{{_{{\bf U}}t_1}}, {_{1{\bf U}}x}, {_{2{\bf U}}y}, {_{1{\bf U}}z, {_{m{\bf U}}x}, {_{{\bf U}}t}_2}\}$ is $l_{{\bf U}, {\bf V}}$-stable in ${\bf U} \approx {\bf V}$.
\end{claim}

\begin{proof} In view of the fact that $yt_1 yp_1^2 \dots p_n^2zt_2z \succeq xxyy$, it is enough to prove that the sets $\{ {{_{{\bf U}}t_1}}, {_{1{\bf U}}x}, {_{2{\bf U}}y} \}$ and $\{ {_{1{\bf U}}z, {_{m{\bf U}}x}, {_{{\bf U}}t}_2} \}$ are $l_{{\bf U}, {\bf V}}$-stable in ${\bf U} \approx {\bf V}$.
We only show that the set $\{{{_{{\bf U}}t_1}}, {_{1{\bf U}}x}, {_{2{\bf U}}y} \}$ is $l_{{\bf U}, {\bf V}}$-stable in ${\bf U} \approx {\bf V}$. (The proof for the other set is symmetric.)

Since the variable $t_1$ is linear in ${\bf U}$,  we may assume that $t_1$ is linear in ${\bf u}$ and  $_{\bf u}t_1 = \Theta_{\bf u}^{-1}(_{\bf U}t_1)$.
Since $y$ occurs twice in ${\bf U}$, we may assume that either $_{1{\bf u}}y = \Theta^{-1}_{\bf u}({_{1{\bf U}}y})$ and $_{2{\bf u}}y =\Theta^{-1}_{\bf u}({_{2{\bf U}}y})$ or $ _{\bf u}t_3 = \Theta^{-1}_{\bf u}({_{1{\bf U}}y})$ and $ _{\bf u}y = \Theta^{-1}_{\bf u}({_{2{\bf U}}y})$.
In view of Fact \ref{thetainv} we may assume that $\Theta^{-1}_{\bf u}({_{1{\bf U}}x}) = {_{1{\bf u}}x}$.

If the variable $x$ occurs $m$ times in ${\bf u}$, then $\Theta^{-1}(x)=\{x \}$, $occ_{\bf v}(x)=m$ and $\Theta^{-1}_{\bf u}({_{m{\bf U}}x}) = {_{m{\bf u}}x}$.
Since ${_{1{\bf u}}x} \le_{\bf u} (_{\bf u}t)  \le_{\bf u} {_{m{\bf u}}x}$,
we have that ${\bf u}(x,y,t_1,t) = yt_1x^dyx^ctx^p$ (or ${\bf u}(x,y,t_1,t) =t_1x^dyx^ctx^p$) for some $d,p>0$ and $d+c+p=m$.
Since for each  $0 <d <m$ the word $yt_1yx^dtx^{m-d}$ is an isoterm for $S$,
we have that $(_{\bf v}t_1) \le_{\bf v} ({_{1{\bf v}}x}) \le_{\bf v} ({_{2{\bf v}}y})$ (or $({_{\bf v}t}_1) <_{\bf v} ({_{1{\bf v}}x}) <_{\bf v} ({_{{\bf v}}y})$). Since $\Theta^{-1}(x)=\{x \}$, we have that $\Theta^{-1}_{\bf v}({_{1{\bf V}}x}) = {_{1{\bf v}}x}$. Since $\Theta^{-1}_{\bf v}$ is a homomorphism from $(\ocs({\bf V}), <_{\bf V})$ to $(\ocs({\bf v}), <_{\bf v})$
and in view of Lemma \ref{samelabels}, we conclude that
$({_{{\bf V}}t_1}) <_{\bf V}  ({_{1{\bf V}}x}) <_{\bf V} ({_{2{\bf V}}y})$.

If the variable $x$ occurs less than $m$ times in ${\bf u}$, then ${\bf u}(x,y,t_1)=yt_1x^dyx^c$ (or ${\bf u}(x,y,t_1)= t_1x^dyx^c$) with $d>0$ and $occ_{\bf u}(x)=d+c<m$.
Since the words $yt_1yx^{d+c}$ and $x^dtx^c$ are isoterms for $S$, we have  that $({_{\bf v}t}_1) <_{\bf v} ({_{1{\bf v}}x}) <_{\bf v} ({_{2{\bf v}}y})$
(or $({_{\bf v}t}_1) <_{\bf v} ({_{1{\bf v}}x}) <_{\bf v} ({_{{\bf v}}y})$).
In view of Fact \ref{thetainv}, for some variable $z$ we have $\Theta^{-1}_{\bf v}({_{1{\bf V}}x}) = {_{1{\bf v}}z}$. If $z=x$ then Lemma \ref{samelabels} implies that
$({_{{\bf V}}t_1}) <_{\bf V}  ({_{1{\bf V}}x}) <_{\bf V} ({_{2{\bf V}}y})$.

 If $z \ne x$ then we have $({_{1{\bf v}}z}) <_{\bf v} ({_{1{\bf v}}x}) \le_{\bf v} ({_{2{\bf v}}y})$ (or $({_{1{\bf v}}z}) <_{\bf v} ({_{1{\bf v}}x}) \le_{\bf v} ({_{{\bf v}}y})$) and $({_{1{\bf u}}x}) <_{\bf u} ({_{1{\bf u}}z})$.
 Let us assume that $({_{1{\bf v}}z})\le_{\bf v} ({_{\bf v}t}_1)$ and obtain a contradiction.
 Since  by Claim \ref{r8} we have $occ_{\bf u}(z) =  occ_{\bf v}(z) < m$, the word ${\bf v}(z,t_1)$ is an isoterm for $S$.
 Consequently, $({_{1{\bf u}}z}) \le_{\bf u} ({_{\bf u}t}_1)$. But on the other hand $(_{\bf u}t_1) \le_{\bf u} ({_{1{\bf v}}x}) <_{\bf u} ({_{1{\bf u}}z})$.
 To avoid a contradiction, we must assume that $({_{\bf v}t}_1) \le_{\bf v} ({_{1{\bf v}}z})$.

Since $\Theta^{-1}_{\bf v}$ is a homomorphism from $(\ocs({\bf V}), <_{\bf V})$ to $(\ocs({\bf v}), <_{\bf v})$ and  $({_{\bf v}t}_1) \le_{\bf v}({_{1{\bf v}}z}) \le_{\bf v} ({_{2{\bf v}}y})$ (or $({_{\bf v}t}_1) \le_{\bf v}({_{1{\bf v}}z}) \le_{\bf v} ({_{{\bf v}}y})$) we have that $({_{{\bf V}}t_1}) <_{\bf V}  ({_{1{\bf V}}x}) <_{\bf V} ({_{2{\bf V}}y})$.
\end{proof}

Therefore, the monoid $S$ is non-finitely based  by Lemma \ref{nfblemma1}.
\end{proof}

In articles \cite{OS1,OS2}, we will use the sufficient conditions in Table \ref{classes} to identify certain intervals in the lattice of semigroup varieties
which contain only NFB monoids. The results from \cite{OS1,OS2} also imply that if $W$ is one of the eight sets of words in the left column of Table \ref{classes} then the finite monoid $S(W)$ satisfies the corresponding identities in the right column, and consequently, is non-finitely based by Theorem \ref{nfbpairs}.
Two of these monoids, $S(\{xyyx\})$ and $S(\{xtxyty\})$ have been known to be non-finitely based since \cite{MJ,JS,OS}.

A semigroup can be non-finitely based by way of multiple sets of identities. For example, according to the first row of Table \ref{classes}, the monoid
$S(\{xyyx\})$ is non-finitely based by way of the set of identities $\{ xx[Yn][nY] \approx [Yn][nY]xx \mid n>1\}$, while according  to Lemma 5.2 in \cite{OS} the monoid
$S(\{xyyx\})$ is non-finitely based by way of the set of identities $\{ x[Yn]tx[nY] \approx [Yn]xt[nY]x \mid n>1 \}$. According to Example 4.1 in \cite{JS} the monoid
$S(\{xyyx\})$ is non-finitely based by way of yet another set of identities. Another example: the original proof of M. Jackson in \cite{MJ} that the monoid $S(\{xtxyty\})$ is non-finitely based uses a set of identities different from the one in the third row of Table \ref{classes}.

\section{Simplified versions of Perkins' and Lee's sufficient conditions under which a semigroup is non-finitely based}

The following theorem is obtained from Perkins sufficient condition \cite[Theorem 7]{P} by dropping its requirement that ``$S$ does not satisfy the identity $xxt \approx txx$".

\begin{theorem} \label{PSC}

Let $S$ be a monoid such that

(i) $S$ satisfies the identity
${\bf U}_n= x [Yn]x [nY] \approx  x [nY] x [Yn] ={\bf V}_n $  for each $n>1$;

(ii) $S$ does not satisfy the identity $xyxy \approx xyyx$;

(iii) the words $xytyx$ and $xtyxy$ are isoterms for $S$.

Then $S$ is non-finitely based.

\end{theorem}

\begin{proof} Take $n>1$ and consider the identity ${\bf U}_n \approx {\bf V}_n$. For each $z \in \{y_2, y_3, \dots, y_{n-1}, y_{n} \}$ we choose
$X_i=\{ {_{1{\bf U}_n}x},  {_{2{\bf U}_n}x},  {_{1{\bf U}_n}y}_1,  {_{2{\bf U}_n}y}_1,  {_{1{\bf U}_n}z}, {_{2{\bf U}_n}z} \}$. We also choose $\mathfrak X = \con ({\bf U}_n)$. Evidently, the set $X=X_2$ is $l_{{\bf U}_n, {\bf V}_n}$-unstable in ${\bf U}_n \approx {\bf V}_n$.

Let us check the third condition of Lemma \ref{nfblemma1}. Let ${\bf U} \in [\![{\bf U}_n]\!]_S$ be a word such that each variable is 2-occurring in $\bf U$ and
for each $2 \le i \le n$ the set $X_i$ is $l_{{\bf U}_n, {\bf U}}$-stable in ${\bf U}_n \approx {\bf U}$. So, for each $z \in \{y_2, y_3, \dots, y_{n-1}, y_{n} \}$ we have

\[({_{1{\bf U}}x}) <_{\bf U} ({_{1{\bf U}}y}_1) <_{\bf U} ({_{1{\bf U}}z}) <_{\bf U} ({_{2{\bf U}}x})  <_{\bf U} ({_{2{\bf U}}z}) <_{\bf U} ({_{2{\bf U}}y}_1).\]

Let ${\bf u}$ be a word in less than $n/4$ variables such that $\Theta({\bf u})={\bf U}$ for some substitution $\Theta: \mathfrak A \rightarrow \mathfrak A^+$.
Since the word ${\bf u}$ has less than $n/2$ variables, for some
$c_1 \in \ocs({\bf u})$ and $2 \le i, j \le n$ both ${_{1{\bf U}}y}_{i}$ and ${_{1{\bf U}}y}_{j}$ are contained in $\Theta_{\bf u}(c_1)$ and for some
$c_2 \in \ocs({\bf u})$ and $2 < k, p < n$ both ${_{2{\bf U}}y}_{k}$ and ${_{2{\bf U}}y}_{p}$ are contained in $\Theta_{\bf u}(c_2)$.

Since $S$ satisfies $xyyx \approx yxxy$ but does not satisfy $xyxy \approx xyyx$, the monoid $S$
satisfies neither $xyxy \approx xyyx$ nor $xyxy \approx yxxy$. Therefore,
both $c_1$ and $c_2$ must be the only occurrences of some linear variables $t_1$ and $t_2$ in ${\bf u}$.

Now fix an arbitrary $z \in \{y_2, y_3, \dots, y_{n-1}, y_{n} \}$.
Since $\Theta^{-1}_{\bf u}$ is a homomorphism from $(\ocs({\bf U}), <_{\bf U})$ to $(\ocs({\bf u}), <_{\bf u})$, for each $z \in \{y_2, y_3, \dots, y_{n-1}, y_{n} \}$
one of the following four cases is possible:

$(1) \hskip .1 in  \Theta^{-1}_{\bf u}({_{1{\bf U}}x})  \le_{\bf u}  \Theta^{-1}_{\bf u}({_{1{\bf U}}y}_1) \le_{\bf u} (_{\bf u}t_1)
  \le_{\bf u} \Theta^{-1}_{\bf u}({_{1{\bf U}}z})  \le_{\bf u}$

\( \begin{array}{l}
  \le_{\bf u}  \Theta^{-1}_{\bf u}({_{2{\bf U}}x}) \le_{\bf u} (_{\bf u}t_2) \le_{\bf u} \Theta^{-1}_{\bf u}({_{2{\bf U}}z})  \le_{\bf u} \Theta^{-1}_{\bf u}({_{2{\bf U}}y}_1);

\end{array} \)

\hskip .2in

$(2) \hskip .1 in  \Theta^{-1}_{\bf u}({_{1{\bf U}}x})  \le_{\bf u}  \Theta^{-1}_{\bf u}({_{1{\bf U}}y}_1) \le_{\bf u} \Theta^{-1}_{\bf u}({_{1{\bf U}}z})
  \le_{\bf u}  (_{\bf u}t_1) \le_{\bf u}$

\( \begin{array}{l}
  \le_{\bf u}  \Theta^{-1}_{\bf u}({_{2{\bf U}}x}) \le_{\bf u} (_{\bf u}t_2) \le_{\bf u} \Theta^{-1}_{\bf u}({_{2{\bf U}}z})  \le_{\bf u} \Theta^{-1}_{\bf u}({_{2{\bf U}}y}_1);

\end{array} \)

\hskip .2in

$(3) \hskip .1 in  \Theta^{-1}_{\bf u}({_{1{\bf U}}x})  \le_{\bf u}  \Theta^{-1}_{\bf u}({_{1{\bf U}}y}_1) \le_{\bf u} (_{\bf u}t_1)
  \le_{\bf u} \Theta^{-1}_{\bf u}({_{1{\bf U}}z})  \le_{\bf u}$

\( \begin{array}{l}
  \le_{\bf u}  \Theta^{-1}_{\bf u}({_{2{\bf U}}x}) \le_{\bf u} \Theta^{-1}_{\bf u}({_{2{\bf U}}z}) \le_{\bf u} (_{\bf u}t_2)   \le_{\bf u} \Theta^{-1}_{\bf u}({_{2{\bf U}}y}_1);

\end{array} \)

\hskip .2in

$(4) \hskip .1 in  \Theta^{-1}_{\bf u}({_{1{\bf U}}x})  \le_{\bf u}  \Theta^{-1}_{\bf u}({_{1{\bf U}}y}_1) \le_{\bf u} \Theta^{-1}_{\bf u}({_{1{\bf U}}z})
  \le_{\bf u}  (_{\bf u}t_1) \le_{\bf u}$

\( \begin{array}{l}
  \le_{\bf u}  \Theta^{-1}_{\bf u}({_{2{\bf U}}x}) \le_{\bf u} \Theta^{-1}_{\bf u}({_{2{\bf U}}z}) \le_{\bf u} (_{\bf u}t_2)   \le_{\bf u} \Theta^{-1}_{\bf u}({_{2{\bf U}}y}_1).

\end{array} \)

\hskip .2in

In each of the four cases, $\Theta^{-1}_{\bf u}({_{1{\bf U}}x})$, $\Theta^{-1}_{\bf u}({_{1{\bf U}}y}_1)$ and $\Theta^{-1}_{\bf u}({_{1{\bf U}}z})$,
are occurrences of some variables $x$, $y$ and $z$ in ${\bf u}$. By Fact \ref{2occ}, each of the variables $x$, $y$ and $z$ occurs at most twice in ${\bf u}$ and
if $occ_{\bf u}(x)= occ_{\bf u}(y) = occ_{\bf u}(z) =2$ then  $_{1{\bf u}}x = \Theta^{-1}_{\bf u}({_{1{\bf U}}x})$, $_{2{\bf u}}x =\Theta^{-1}_{\bf u}({_{2{\bf U}}x})$, $_{1{\bf u}}y = \Theta^{-1}_{\bf u}({_{1{\bf U}}y}_1)$, $_{2{\bf u}}y =\Theta^{-1}_{\bf u}({_{2{\bf U}}y}_1)$, $_{1{\bf u}}z = \Theta^{-1}_{\bf u}({_{1{\bf U}}z})$ and $_{2{\bf u}}z =\Theta^{-1}_{\bf u}({_{2{\bf U}}z})$.

There are four possibilities for the word ${\bf u}(x,y,z,t_1, t_2)$:

(1) $\hskip .1 in {\bf u}(x,y,z,t_1, t_2) = xyt_1zxt_2zy$;

(2) $\hskip .1 in {\bf u}(x,y,z,t_1, t_2) = xy zt_1xt_2zy$;

(3) $\hskip .1 in {\bf u}(x,y,z,t_1, t_2) = xyt_1z x z t_2 y$;

(4) $ \hskip .1 in {\bf u}(x,y,z,t_1, t_2) = xy zt_1x zt_2y$.

Since the word $xtx$ is an isoterm for $S$, the words ${\bf u}(x,t_1, t_2)$, ${\bf u}(y,t_1, t_2)$, ${\bf u}(z,t_1, t_2)$ are also isoterms for $S$.
In particular, the variables $x$, $y$ and $z$ are stable in $\bf u$ with respect to $S$. Since $xytyx \preceq xytytx \sim xytxty$,
 the pair $\{x,y\}$ is stable in ${\bf u}(x,y,z,t_1, t_2)$  with respect to $S$ in each of the four cases.

(1) $\hskip .1 in$ If ${\bf u}(x,y,z,t_1, t_2) = xyt_1zxt_2zy$ then
the pair $\{y,z\}$ is stable because the word
 $yztzy$ is an isoterm for $S$  and the pair $\{x,z\}$ is stable because the word $xtzxz$ is an isoterm for $S$.

(2) $\hskip .1 in$ If ${\bf u}(x,y,z,t_1, t_2)= xy zt_1xt_2zy$ then
the pair $\{y,z\}$ is stable because the word $yztzy$ is an isoterm for $S$.

(3) $\hskip .1 in$ If ${\bf u}(x,y,z,t_1, t_2)= xyt_1z x z t_2 y$ then
the pair $\{x,z\}$ is stable because the word $xtzxz$ is an isoterm for $S$.

(4) $\hskip .1 in$ If ${\bf u}(x,y,z,t_1, t_2)= xy zt_1x zt_2y$ then
the pair $\{y,z\}$ is stable because the word $yztzy$ is an isoterm for $S$ and
 the pair $\{x,z\}$ is stable because the word $xztzx$ is an isoterm for $S$.

In each of the four cases, Fact \ref{st1} implies that the word ${\bf u}(x,y,z,t_1, t_2)$ is an isoterm for $S$.
If one of the variables $\{x,y,z\}$ is linear in $\bf u$ then one can show that ${\bf u}(x,y,z,t_1, t_2)$ is an isoterm for $S$ by using similar but more simple arguments.

Now let $\bf v$ be an arbitrary word for which $S$ satisfies the identity ${\bf u} \approx {\bf v}$ and ${\bf V} = \Theta({\bf v})$.
Since each variable in $\Theta^{-1}(x) \cup \Theta^{-1}(y) \cup \Theta^{-1}(z)$ is stable in ${\bf u} \approx {\bf v}$, each variable in $\{x,y,z\}$
is stable in ${\bf U} \approx {\bf V}$, i.e. $occ_{\bf V}(x)= occ_{\bf V}(y) = occ_{\bf V}(z) = 2$.
Now Lemma \ref{samelabels} implies that

\[({_{1{\bf V}}x}) <_{\bf V} ({_{1{\bf V}}y}_1) <_{\bf V} ({_{1{\bf V}}z}) <_{\bf V} ({_{2{\bf V}}x})  <_{\bf V} ({_{2{\bf V}}z}) <_{\bf V} ({_{2{\bf V}}y}_1).\]

This means that  each variable in
$\bf U$ is stable in ${\bf U} \approx {\bf V}$ and for each $2 \le i \le n$, the set $l_{{\bf U}_n, {\bf U}}(X_i)$ is $l_{{\bf U}, {\bf V}}$-stable in ${\bf U} \approx {\bf V}$. Therefore, the monoid $S$ is non-finitely based  by Lemma \ref{nfblemma1}.
\end{proof}

Since Theorem 7 in \cite{P} implies that the six-element Brandt monoid is non-finitely based,
Theorem \ref{PSC} implies that as well.

Theorem \ref{EL} below gives a sufficient condition under which a semigroup is non-finitely based. It is obtained from Theorem 1 in \cite{EL2} by dropping one of its requirements. Theorem \ref{EL} contains the only existing sufficient condition of form ($\mathcal P$) (see the introduction) where a semigroup $S$ is not required to be a monoid. In fact, it is easy to check that a monoid cannot satisfy all three requirements of Theorem \ref{EL}.

\begin{theorem} \label{EL}  Let $k >1$ be any fixed integer. Suppose that a semigroup $S$

(i) satisfies the identities

\[{\bf U}_n= xy^k_1y^k_2 \dots y^k_nx  \approx  xy^k_ny^k_{n-1} \dots y^k_1x ={\bf V}_n, n>1;\]

(ii) does not satisfy the identity

\begin{equation} \label{e1}
x^ky^kx^k \approx x^k(y^kx^k)^{k+1};
\end{equation}

(iii) satisfies the identities

\begin{equation} \label{e2}
x^{k+2} \approx x^2, x^{k+1}yx \approx xyx,  xyx^{k+1} \approx xyx.
\end{equation}

Then $S$ is non-finitely based.

\end{theorem}

We precede the proof of Theorem \ref{EL} by two lemmas from \cite{EL2}. The first lemma is similar to the first sentence in the proof of Theorem \ref{SL1}.

\begin{lemma} \label{P} \cite[Lemma 9]{EL2} Let $S$ be a semigroup that satisfies Conditions (ii) and (iii) in Theorem \ref{EL}.
If ${\bf u} \approx {\bf v}$ is an identity of $S$ then $\lin ({\bf u}) = \lin ({\bf v})$ and $\non ({\bf u}) = \non ({\bf v})$.
\end{lemma}

Lemma 13 in \cite{EL2} describes the equivalence class of $[\![{\bf U}_n]\!]_S$. The following lemma is an immediate consequence of Lemma 13 in \cite{EL2}.

\begin{lemma} \label{K} Let $S$ be a semigroup that satisfies Conditions (ii) and (iii) in Theorem \ref{EL}.
Let $n>2$ and ${\bf U} \in [\![{\bf U}_n]\!]_S$.
If for some $1 \le i <n$ we have ${_{last{\bf U}}y}_i <_{\bf U} {_{1{\bf U}}y}_{i+1}$  then  for each $1 \le i <n$ we have ${_{last{\bf U}}y}_i <_{\bf U} {_{1{\bf U}}y}_{i+1}$.
\end{lemma}

Let $e_{{\bf u}, {\bf v}}$ be a map from $\{ _{1{\bf u}}x, \ _{last{\bf u}}x \mid x \in \non({\bf u}) \cap \non({\bf v}) \}$ to $\{ _{1{\bf v}}x, \ _{last{\bf v}}x \mid x \in \non({\bf u}) \cap \non({\bf v}) \}$ defined by $e_{{\bf u}, {\bf v}} (_{1{\bf u}}x) = {_{1{\bf v}}x}$ and $e_{{\bf u}, {\bf v}} (_{last{\bf u}}x) = {_{last{\bf v}}x}$. For example, if ${\bf u} = zxyzyyx$ and ${\bf v} = yxxyxpp$ then $e_{{\bf u}, {\bf v}}$ is a bijection between $\{_{1{\bf u}}x, {_{2{\bf u}}x}, {_{1{\bf u}}y}, {_{3{\bf u}}y} \}$ and $\{_{1{\bf v}}x, {_{3{\bf v}}x}, {_{1{\bf v}}y}, {_{2{\bf v}}y} \}$. The set $\{_{1{\bf u}}x, {_{1{\bf u}}y} \}$
is $e_{{\bf u}, {\bf v}}$ - unstable in ${\bf u} \approx {\bf v}$ but the set $\{_{2{\bf u}}x, {_{3{\bf u}}y} \}$
is $e_{{\bf u}, {\bf v}}$ - stable in ${\bf u} \approx {\bf v}$.

It is easy to check that the set $\{ e_{{\bf u}, {\bf v}} \mid ({\bf u} \approx {\bf v}) \in \mathfrak A^+ \times \mathfrak A^+, \non({\bf u}) \cap \non({\bf v}) \ne \emptyset \}$ satisfies all three conditions of Definition \ref{trans}, and consequently is a good collection of maps.

{\bf Proof of Theorem \ref{EL}.}

Take $n>4$ and consider the identity ${\bf U}_n \approx {\bf V}_n$ and set

 $X=\{{_{last{\bf U}_n}y}_1, {_{1{\bf U}_n}y}_2, {_{last{\bf U}_n}y}_2, {_{1{\bf U}_n}y}_3, {_{last{\bf U}_n}y}_3, {_{1{\bf U}_n}y}_4, \dots, {_{last{\bf U}_n}y}_{n-1},  {_{1{\bf U}_n}y}_n \}$. We choose $\mathfrak X$ to be the empty set of variables.

First notice that the map $e_{{\bf U}_n, {\bf U}_n}$ is defined on $X$ because $\non({\bf U}_n) = \con({\bf U}_n)$.
Evidently, the set $X$ is  $e_{{\bf U}_n, {\bf V}_n}$-unstable in ${\bf U}_n \approx {\bf V}_n$.

Let us check the third condition of Lemma \ref{nfblemma1}. Let ${\bf U} \in [\![{\bf U}_n]\!]_S$ be a word such that the set $X$ is  $e_{{\bf U}_n, {\bf U}}$-stable in ${\bf U}_n \approx {\bf U}$.

Let ${\bf u}$ be a word in less than $n/4$ variables such that $\Theta({\bf u})={\bf U}$ for some substitution $\Theta: \mathfrak A \rightarrow \mathfrak A^+$.
Since the word ${\bf u}$ has less than $n/4$ variables, for some variable $t \in \con({\bf u})$ the word $\Theta({\bf u})$ contains occurrences of at least four distinct variables in $ \{y_1, y_2, \dots, y_n\}$. Since the set $X$ is  $e_{{\bf U}_n, {\bf U}}$-stable in ${\bf U}_n \approx {\bf U}$, the variable $t$ is linear in $\bf u$ and for some $1 \le i < n-1$ the interval $\Theta_{\bf u}(t)$ contains the following substructure:

 $\{ ({_{1{\bf U}}y}_i)  <_{\bf U} ({_{last{\bf U}}y}_i)  <_{\bf U}  ({_{1{\bf U}}y}_{i+1})   <_{\bf U} ({_{last{\bf U}}y}_{i+1}) \}$.

Now let $\bf v$ be an arbitrary word for which $S$ satisfies the identity ${\bf u} \approx {\bf v}$ and ${\bf V} = \Theta({\bf v})$.
In view of Lemma \ref{P}, the variable $t$ is linear in $\bf v$.
Therefore, the interval $\Theta_{\bf v}(t)$ contains the following substructure:
$\{ ({_{1{\bf V}}y}_i)  <_{\bf V} ({_{last{\bf V}}y}_i)  <_{\bf V}  ({_{1{\bf V}}y}_{i+1})   <_{\bf V} ({_{last{\bf V}}y}_{i+1}) \}$.

So, by Lemma \ref{K} the set $e_{{\bf U}_n, {\bf U}}(X)$ is $e_{{\bf U}, {\bf V}}$-stable in ${\bf U} \approx {\bf V}$. Therefore, the monoid $S$ is non-finitely based  by Lemma \ref{nfblemma1}. $\Box$

\vskip .2in

 Theorem 1 in \cite{EL2} is obtained from Theorem 8 in \cite{WTZ1} by extracting the properties of the semigroup $L$ which are responsible for $L$ being non-finitely based by way of the set of identities $\{xy^2_1y^2_2 \dots y^2_nx  \approx  xy^2_ny^2_{n-1} \dots y^2_1x \mid  n>1 \}$. Since Theorem \ref{EL} reduces the number of these properties, it immediately implies the following.

\begin{cor} \label{WT1} \cite[Theorem 8]{WTZ1} The six-element semigroup $L= \langle a,b \mid aa=a, bb=b, aba=0 \rangle$ is non-finitely based.
\end{cor}

\section{A sufficient condition that implies that most Straubing's monoids are non-finitely based}

We say that a word ${\bf u}$ is a {\em scattered subword} of a word ${\bf v}$ whenever there exist words ${\bf u}_1, \dots, {\bf u}_k, {\bf v}_0, {\bf v}_1, \dots, , {\bf v}_{k-1}, {\bf v}_k \in \mathfrak A^*$ such that \vskip .1in

${\bf u} = {\bf u}_1 \dots {\bf u}_k$ and ${\bf v} = {\bf v}_0 {\bf u}_1 {\bf v}_1\dots {\bf v}_{k-1}{\bf u}_k{\bf v}_k$;
\vskip .1in
in other terms, this means that one can extract $\bf u$ treated as a sequence of letters from the sequence $\bf v$.
We denote by $J_m$ the set of all identities $({\bf u} \approx {\bf v})$ such that the words $\bf u$ and $\bf v$ have the same set of scattered subwords of length $\le m$.

We say that a set of identities $\Sigma$ is {\em finitely based} if all identities in $\Sigma$ can be derived from a finite subset of identities in $\Sigma$.
This section as well as Section 7 in \cite{OS3} is motivated by the following result of F. Blanchet-Sadri.

\begin{lemma} \cite[Theorem 3.4]{BS} \label{b-s1} The set of identities $J_m$ is finitely based if and only if $m \le 3$.

\end{lemma}

 In \cite{MV1}, M. Volkov proved that for each $m \ge 1$, the set $J_m$ is the equational theory of each of the following (finite) monoids $S_{m+1}$:

$\bullet$ The monoid of all reflexive binary relations on a set with $m+1$ elements;

$\bullet$ The monoid of all $(m+1) \times (m+1)$ upper triangular matrices over the Boolean semiring;

$\bullet$ The monoid of all order preserving and extensive transformations of a chain with $m+1$ elements.

M. Volkov refers to these monoids as Straubing monoids, because these monoids were studied by H. Straubing \cite{St} in connection to the celebrated Simon's
theorem \cite{Si} about the finite monoids which recognize the piecewise testable languages.
In view of the famous Eilenberg correspondence (\cite{El}, see also \cite{Pin}), Theorem 2 in \cite{MV1} says that each Straubing monoid $S_{m+1}$ generates the pseudovariety of piecewise $m$-testable languages.

In view of the result of Volkov, the ``only if" part of Lemma \ref{b-s1} can be reformulated as follows.

\begin{lemma} \cite{BS} \label{b-s} For each $m \ge 4$ each of the three Straubing monoids $S_{m+1}$ is non-finitely based by way of $\{x^{m-2}  [ZYn] x [Yn][Zn] x \approx x^{m-1}  [ZYn] x [Yn][Zn] x \mid n > 1 \}$.
\end{lemma}

Lemma \ref{b-s} motivates the following sufficient condition under which a monoid is non-finitely based.

\begin{theorem} \label{BSnew} Let $m \ge 3$ be any fixed integer. Let $S$ be a monoid such that

(i) $S$ satisfies the identity
\[{\bf U}_n = x^{m-2} [Yn]x[nY]x \approx x^{m-1} [Yn]x[nY]x = {\bf V}_n\] for each $n>1$;

(ii) the words $xyyx$, $x^{m-1}$ are isoterms for $S$;

(iii) the word $x^{m-2}t_1xt_2x$ is an isoterm for $S$;

(iv) if $S$ satisfies an identity of the form ${\bf u} \approx x^{m-2}yxyx$ and $occ_{\bf u}(x)=m$ then ${\bf u} = x^{m-2}yxyx$.

Then $S$ is non-finitely based.

\end{theorem}

\begin{proof} Since the word $x^{m-1}$ is an isoterm for $S$, the monoid $S$ satisfies only regular identities.
First, we need the following property of the equivalence class $[\![{\bf U}_n]\!]_S$.

\begin{claim} \label{mclass} If ${\bf U} \in [\![{\bf U}_n]\!]_S$ and $occ_{\bf U}(x) = m$ then ${\bf U} = {\bf U}_n$.

\end{claim}

\begin{proof} \label{bs1} Since the word $xyyx$ is an isoterm for $S$ the word $D_x({\bf U}) = [Yn][nY] =y_1y_2 \dots y_ny_n \dots y_2y_1$ is also an isoterm for $S$. The rest follows from Condition (iv).\end{proof}

Now take $n>10$ and $\mathcal X = \{ x\}$.  Evidently, the variable $x$ is unstable in ${\bf U}_n \approx {\bf V}_n$.

Let us check the third condition of Lemma \ref{nfblemma1}. Let ${\bf U} \in [\![{\bf U}_n]\!]_S$ be a word such that $occ_{\bf U}(x) = m$.
Then by Claim \ref{mclass} we have that ${\bf U} = {\bf U}_n$.

Let ${\bf u}$ be a word in less than $n/4$ variables such that $\Theta({\bf u})={\bf U}$ for some substitution $\Theta: \mathfrak A \rightarrow \mathfrak A^+$. Since the word ${\bf u}$ has less than $n/2$ variables, for some $c_1 \in \ocs({\bf u})$ and $1 \le i \le n$ both ${_{1{\bf U}}y}_{i}$ and ${_{1{\bf U}}y}_{i+1}$ are contained in $\Theta_{\bf u}(c_1)$ and for some $c_2 \in \ocs({\bf u})$ and $1 \le j \le n$ both ${_{2{\bf U}}y}_{j+1}$ and ${_{2{\bf U}}y}_{j}$ are contained in $\Theta_{\bf u}(c_2)$. Claim \ref{mclass} implies that $c_1$ and $c_2$ must be the only occurrences of some linear variables $t_1$ and $t_2$ in ${\bf u}$.

Since $\Theta^{-1}_{\bf u}$ is a homomorphism from $(\ocs({\bf U}), <_{\bf U})$ to $(\ocs({\bf u}), <_{\bf u})$,
we have that $\Theta^{-1}_{\bf u}({_{{m-2}{\bf U}}x}) \le_{\bf u} t_1 \le_{\bf u} \Theta^{-1}_{\bf u}({_{{m-1}{\bf U}}x}) \le_{\bf u} (_{\bf u}t_2) \le_{\bf u} \Theta^{-1}_{\bf u}({_{m{\bf U}}x})$.

Now let $\bf v$ be an arbitrary word for which $S$ satisfies the identity ${\bf u} \approx {\bf v}$ and ${\bf V} = \Theta({\bf v})$.
We consider two cases.

{\bf Case 1}: $\Theta^{-1}(x)$ contains some variable that occurs in $\bf u$ less than $m$ times. Since $occ_{\bf U}(x)=m$, each variable in $\Theta^{-1}(x)$ occurs in $\bf u$ less than $m$ times. Since the word $x^{m-1}$ is an isoterm for $S$, each variable in  $\Theta^{-1}(x)$
 is stable in ${\bf u} \approx {\bf v}$. Consequently, the variable $x$ is stable in ${\bf U} \approx {\bf V}$.

{\bf Case 2}: $\Theta^{-1}(x)$ contains only one variable $x$ and $occ_{\bf u}(x)=m$. Then
${\bf u}(x,t_1,t_2) = x^{m-2} t_1 xt_2 x$. Since it is an isoterm for $S$, the variable $x$ is stable in  ${\bf u} \approx {\bf v}$. Consequently, the variable $x$ is stable in ${\bf U} \approx {\bf V}$. Therefore, the monoid $S$ is non-finitely based  by Lemma \ref{nfblemma1}.
\end{proof}

The next corollary yields an alternative proof of the ``only if" part of Theorem 3.4 in \cite{BS}. (See Lemma \ref{b-s1} above.) The ``if" part of
this theorem is reproved in Proposition 4.2 and Theorem 7.2 in \cite{OS3}.

\begin{cor} \label{Em} \cite{BS}
For each $m \ge 4$ each of the three Straubing monoids $S_{m+1}$ is non-finitely based.
\end{cor}

\begin{proof} Fix $m \ge 4$. Let us check all conditions of Theorem \ref{BSnew}.

 (i) Evidently, for each $ n \ge 1$, the words
${\bf U}_n = x^{m-2} [Yn]x[nY]x$ and ${\bf V}_n = x^{m-1}[Yn]x[nY]x$ have the same set of scattered subwords of length $\le m$.

(ii) It is easy to see that the words $x^{m-1}$ and $xyyx$ are isoterms for $S_{m+1}$.

(iii) Let us check that the word $x^{m-2}t_1xt_2x$ is an isoterm for $S_{m+1}$.

Suppose that $S_{m+1}$ satisfies an identity of the form ${\bf u} \approx x^{m-2}t_1xt_2x$. Since the word $xy$ is an isoterm for $S_{m+1}$,
 we have ${\bf u} = x^p t_1 x^d t_2 x^q$ for some $p, d, q \ge 0$. Since $x^{m-2}t_1$ and $t_2x$ are subwords of $x^{m-2}t_1xt_2x$ we have
 $p={m-2}$ and $q=1$. Since $t_1xt_2$ is a subword of $x^{m-2}t_1xt_2x$ we have that $d>0$. Now if $d>1$ then $\bf u$ contains a scattered subword
 $t_1x^2t_2$ which is not a subword of $x^{m-2}t_1xt_2x$. Therefore $d=1$ and ${\bf u} = x^{m-2}t_1xt_2x$.

Condition (iv) can  be easily verified as well. So, the monoid $S_{m+1}$ is non-finitely based by Theorem \ref{BSnew}. \end{proof}

Notice that the monoid $S_{4}$ fails Condition (iii) of Theorem \ref{BSnew}, because it satisfies the identity $xt_1xt_2x \approx xt_1xxt_2x$.
Despite the fact that $S_{4}$ fails only one out of four conditions of Theorem \ref{BSnew}, it is finitely based \cite{BS} (see Lemma \ref{b-s1} above).

\section{New examples of finitely based finite aperiodic monoids whose direct product is non-finitely based}

In 1981, M. Sapir (see \cite{MS2} for exact references) invented a construction of semigroup varieties which produced negative answers to ten traditionally asked questions in the theory of varieties. He suspected that some of these questions may already have negative answers within the class of monoids of the form $S(W)$ where
semigroups are aperiodic and can be easily constructed. Upon his suggestion, M. Jackson and the author \cite{JS} confirmed  that with respect to the finite basis property the class of finite monoids of the form $S(W)$ behaves as badly as the class of all finite semigroups. In particular, the article \cite{JS} contains examples of finite FB aperiodic monoids whose direct product is NFB.

Presently, all existing examples of finite aperiodic FB semigroups whose direct product is NFB are contained in articles \cite{MJ, JS} and involve
monoids of the form $S(W)$ for certain sets of words $W$ with two non-linear variables. (See articles \cite{EL5, MV3} for non-aperiodic examples.)
 The new examples in this section are different from the ones in \cite{MJ, JS} because one of the aperiodic monoids is not of the form $S(W)$.

\begin{ex} \label{FBN} Let $S_4$ denote the monoid of all reflexive binary relations on a set with four elements. Then the direct product $S_4 \times S(\{at_1at_2a\})$ is non-finitely based while both $S_4$ and $S(\{at_1at_2a\})$ are finitely based.

\end{ex}

\begin{proof} As we have already mentioned, the monoid $S_4$ is finitely based by the result of Blanchet-Sadri \cite{BS} and the result of Volkov \cite{MV1}. Since the word $at_1at_2a$ contains only one non-linear variable, the monoid $S(\{at_1at_2a\})$ is finitely based by Theorem 3.2 in \cite{OS}.

As we have already mentioned at the end of Section 6, for $m=3$, the monoid $S_4$ satisfies all conditions of Theorem \ref{BSnew} except for Condition (iii). Now for $m=3$, the monoid $S(\{at_1at_2a\})$ satisfies Condition (i) (because all variables in ${\bf U}_n$ and ${\bf V}_n$ are non-linear) and Condition (iii) of Theorem \ref{BSnew}.
Therefore, the direct product $S_4 \times S(\{at_1at_2a\})$ satisfies all four conditions of Theorem \ref{BSnew} and consequently, is non-finitely based. \end{proof}

We say that a variable $x$ is stable in $\bf u$ with respect to a semigroup $S$ if $x$ is stable in any identity of $S$ of the form ${\bf u} \approx {\bf v}$.

\begin{theorem} \label{BSnew1} Let $m>0$ be any fixed integer. Let $S$ be a monoid such that

(i) $S$ satisfies the identity
\[{\bf U}_n = x^m y_1^2 y_2^2 \dots y_n^2 x \approx x^{m+1} y_1^2 y_2^2 \dots y_n^2 x = {\bf V}_n\] for each $n>1$;

(ii) for each $0 < d < m+1$ the variable $x$ is stable in  $x^{m+1-d}tx^d$ with respect to $S$;

(iii) For each $p,c >1$ the equivalence class $[\![x^py^c]\!]_S$ contains only words of the form $x^iy^j$ for some $i,j >1$;

Then $S$ is non-finitely based.

\end{theorem}

\begin{proof} Take $n>10$ and $\mathcal X = \{ x\}$.
Evidently, the variable $x$ is unstable in ${\bf U}_n \approx {\bf V}_n$.

Let us check the third condition of Lemma \ref{nfblemma1}. Let ${\bf U} \in [\![{\bf U}_n]\!]_S$ be a word such that $occ_{\bf U}(x) = m+1$. Since the word $x$ is an isoterm for $S$, we have $\con ({\bf U}) = \non ({\bf U}) = \con ({\bf U}_n)$.
Condition (iii) implies that

$$({_{1{\bf U}}x}) <_{\bf U} ({_{last{\bf U}}y}_1) <_{\bf U} ({_{1{\bf U}}y}_2) <_{\bf U} ({_{last{\bf U}}y}_2)  <_{\bf U} \dots <_{\bf U} ({_{last{\bf U}}y}_{n-1}) <_{\bf U} ({_{1{\bf U}}y}_n) <_{\bf U} ({_{{(m+1)}{\bf U}}x}).$$

Let ${\bf u}$ be a word in less than $n/4$ variables such that $\Theta({\bf u})={\bf U}$ for some substitution $\Theta: \mathfrak A \rightarrow \mathfrak A^+$. Since the word ${\bf u}$ has less than $n/2$ variables, for some $1 < i < n$ both ${_{last{\bf U}}y}_{i}$ and ${_{1{\bf U}}y}_{i+1}$ are contained in $\Theta_{\bf u}(c)$ where $c$ is the only occurrence of some linear variable $t$ in ${\bf u}$.

Since $\Theta^{-1}_{\bf u}$ is a homomorphism from $(\ocs({\bf U}), <_{\bf U})$ to $(\ocs({\bf u}), <_{\bf u})$,
we have that $\Theta^{-1}_{\bf u}({_{{1}{\bf U}}x}) \le_{\bf u} t \le_{\bf u} \Theta^{-1}_{\bf u}({_{{2}{\bf U}}x})$.

Now let $\bf v$ be an arbitrary word for which $S$ satisfies the identity ${\bf u} \approx {\bf v}$ and ${\bf V} = \Theta({\bf v})$.
We consider two cases.

{\bf Case 1}: $\Theta^{-1}(x)$ contains some variable that occurs in $\bf u$ less than $m+1$ times. Since $occ_{\bf U}(x)=m+1$, each variable in $\Theta^{-1}(x)$ occurs in $\bf u$ less than $m$ times. Since the word $x^{m}$ is an isoterm for $S$, each variable in  $\Theta^{-1}(x)$
 is stable in ${\bf u} \approx {\bf v}$. Consequently, the variable $x$ is stable in ${\bf U} \approx {\bf V}$.

{\bf Case 2}: $\Theta^{-1}(x)$ contains only one variable $x$ and $occ_{\bf u}(x)=m+1$. Then
${\bf u}(x,t) = x^{m+1-d}tx^d$ for some $0 < d < m+1$. By Condition (ii), the variable $x$ is stable in  ${\bf u} \approx {\bf v}$. Consequently, the variable $x$ is stable in ${\bf U} \approx {\bf V}$. Therefore, the monoid $S$ is non-finitely based  by Lemma \ref{nfblemma1}.
\end{proof}

 Let $A_0^1$ denote the monoid obtained by adjoining an identity element to the semigroup $A_0= \langle a,b \mid aa=a, bb=b, ab=0 \rangle$ of order four.
 Comparing the finite basis for $A_0^1$ in \cite[Proposition 3.2(a)]{Ed} with the finite basis for $J_2$ in \cite[Theorem 3.5 ]{BS1} it is easy to conclude
 that $J_2$ is the equational theory of $A_0^1$. (See more details in Proposition 4.2 in \cite{OS3}).
 Since $A_0^1$ is not equationally equivalent to any  monoid with less than five elements \cite{Ed1}, the monoid $A_0^1$ is the smallest monoid whose equational theory is $J_2$. The next statement is reversed in \cite[Theorem 4.8]{OS3}.

\begin{cor} \label{alg} Let $W$ be a set of words with the property
that every adjacent pair of occurrences (if any) of two non-linear variables $x \ne y$ in each word in $W$ is of the form $\{{_{1{\bf u}}x}, {_{last{\bf u}}y} \}$.

Let $m>0$ be the maximal integer for which there is $a \in \mathfrak A$ such that $a^m$ is a subword of a word in $W$.
If $m$ is finite and for each $0 < d < m+1$ the word $b^{m+1-d}{\bf T} b^d$ is a subword of a word in $W$ for some $b \in \mathfrak A$ and ${\bf T} \in \mathfrak A^+$, then the direct product $A^1_0 \times S(W)$ is non-finitely based.
\end{cor}

\begin{proof} The monoid $A^1_0$ satisfies all conditions of Theorem \ref{BSnew1} except for Condition (ii). The monoid $S(W)$ satisfies Conditions (i) and (ii) of Theorem \ref{BSnew1}. Therefore, the direct product $A^1_0 \times S(W)$ satisfies all three conditions of Theorem \ref{BSnew1} and consequently, is non-finitely based. \end{proof}

\begin{ex} \label{last} Denote $W_1 = \{ata\}$, $W_2 = \{a^2ta, ata^2\}$,
$W_3 = \{a^3ta, ata^3, a^2ta^2\}$, $W_4 = \{a^4ta, ata^4, a^3ta^2, a^2ta^3\}$, $\dots$.

Then for each $i=1,2, \dots$ the direct product $A_0^1 \times S(W_i)$ is non-finitely based while both $A_0^1$ and $S(W_i)$ are finitely based.
\end{ex}

\begin{proof} The monoid $A_0^1$ is finitely based by the result of C. Edmunds \cite{Ed}. Since for each $i=1,2, \dots$ the set $W_i$ contains only almost-linear words, the monoid $S(W_i)$ is finitely based by Theorem 3.2 in \cite{OS}. The direct product $A^1_0 \times S(W_i)$ is non-finitely based by Corollary
 \ref{alg}.
\end{proof}

Notice that the non-finitely based monoid $A_0^1 \times S(\{ata\})$ is the smallest existing example of an NFB monoid which is the direct product of two FB monoids.
In particular, it is smaller than the non-finitely based monoid $G_3 \times S(\{ata\})$ \cite{EL5} where $G_3$ is the symmetric group of order six.

We summarize that Lemma \ref{nfblemma1} can be used to verify the non-finite basis property in a wide range of finite aperiodic semigroups. However, there is probably more sophisticated syntactic mechanism behind the non-finite basis property of
the 55-element aperiodic semigroup in \cite{JM}.

\subsection*{Acknowledgement} The author is indebted to an anonymous referee for reading the paper
very carefully with lots of valuable comments. The author also thanks Gili Golan, Edmond Lee and Lev Shneerson for a number of useful
suggestions on improvement of this article and Marcel Jackson for directing her attention to the articles \cite{BS,MV1}.

\end{document}